\documentclass[12pt]{amsart}
\usepackage{amssymb,amsmath,amsfonts,epsfig,latexsym, xypic, hyperref}

\usepackage{mathbbol}

\DeclareSymbolFontAlphabet{\mathbbm}{bbold}
\DeclareSymbolFontAlphabet{\mathbb}{AMSb}%

\usepackage{thmtools}
\usepackage{thm-restate}

\usepackage{enumerate}
\usepackage{tikz}

\tikzset{my_dot/.style={fill, circle, inner sep=0pt,minimum size=3pt}}
\tikzset{my_node/.style={fill, circle, inner sep=0pt,minimum size=3pt}}
\tikzset{inv/.style={fill, circle, inner sep=0pt,minimum size=0pt}}

\newtheorem{theorem}{Theorem}[section]
\newtheorem{definition}[theorem]{Definition}
\newtheorem{proposition}[theorem]{Proposition}
\newtheorem{lemma}[theorem]{Lemma}
\newtheorem{corollary}[theorem]{Corollary}

\newtheorem*{theorem*}{Theorem}
\newtheorem*{corollary*}{Corollary}

\numberwithin{equation}{subsection}

\theoremstyle{definition}
\newtheorem{remark}[theorem]{Remark}
\newtheorem{example}[theorem]{Example}
\newtheorem{observation}[theorem]{Observation}

\def\C{\ensuremath{\mathbb{C}}}

\def\Q{\ensuremath{\mathbb{Q}}}
\def\R{\ensuremath{\mathbb{R}}}
\def\Z{\ensuremath{\mathbb{Z}}}

\def\cM{\ensuremath{\mathcal{M}}}

\DeclareMathOperator{\Aut}{Aut}

\DeclareMathOperator{\Gr}{Gr}

\DeclareMathOperator{\Out}{Out}

\DeclareMathOperator{\val}{val}

\DeclareMathOperator{\Bl}{Bl}

\def\col{\colon}
\def\Dg{\Delta_{g}}
\def\Dgn{\Delta_{g,n}}
\def\Dn{\Delta_{1,n}}
\def\Gmw{{\bf G}}

\def\J{\mathbbm{\Gamma}}

\def\Mgn{M^{\trop}_{g,n}} 
\def\ocM{\overline{\cM}}

\def\rep{\Delta^{\mathrm{rep}}_{g,n}}
\def\trop{\mathrm{trop}}
\def\br{\Delta^{\mathrm{br}}_{g,n}}
\def\lw{\Delta^{\mathrm{lw}}_{g,n}}
\def\weight{\Delta^{\mathrm{w}}_{g,n}}
\def\Ggn{G^{(g,n)}}
\def\bfB{\mathbf{B}}
\def\bfv{v}
\def\Simp{\mathrm{Simp}}

\newcommand{\double}{\genfrac..{0pt}1
{\raise -2pt\hbox{$\scriptstyle\longrightarrow$}}{\raise 4pt\hbox
{$\scriptstyle\longrightarrow$}}}

\newcommand{\FFF}[1]{\widetilde{D}([#1])}

\newcommand{\EE}{\widetilde{D}}
\newcommand{\EEE}[1]{\widetilde{D}^{[#1]}}

\begin{document}

\title{Topology of moduli spaces of tropical curves with marked points}

\author{Melody Chan}
\email{melody\_chan@brown.edu}

\author{S{\o}ren Galatius}
\email{galatius@math.ku.dk}

\author{Sam Payne}
\email{sampayne@utexas.edu}

\begin{abstract}
This article is a sequel to \cite{CGP1-JAMS}.  We study a space $\Dgn$ of genus $g$ stable, $n$-marked tropical curves with total edge length $1$.  Its rational homology is identified with both top-weight cohomology of the complex moduli space $\cM_{g,n}$ and with the homology of a marked version of Kontsevich's graph complex, up to a shift in degrees.

We prove a contractibility criterion that applies to various large subspaces of $\Dgn$.  From this we derive a description of the homotopy type of $\Delta_{1,n}$, the top weight cohomology of $\cM_{1,n}$ as an $S_n$-representation, and additional calculations of $H_i(\Delta_{g,n};\Q)$ for small $(g,n)$. We also deduce a vanishing theorem for homology of marked graph complexes from vanishing of cohomology of $\cM_{g,n}$ in appropriate degrees, by relating both to $\Dgn$. We comment on stability phenomena, or lack thereof.
 \end{abstract}

\maketitle

\begin{center}{\em Dedicated to William Fulton on the occasion of his 80th birthday}\end{center}
\tableofcontents

\section{Introduction}

In \cite{CGP1-JAMS}, we studied the topology of the {\em tropical moduli space} $\Dg$ of stable tropical curves of genus $g$ and total edge length $1$. Here, a tropical curve is a metric graph with nonnegative integer vertex weights; it is said to be {\em stable} if every vertex of weight zero has valence at least 3. With appropriate degree shift, the rational homology of $\Dg$ is isomorphic to both Kontsevich's graph homology and the top weight cohomology of the complex algebraic moduli space $\cM_g$.  As an application of these identifications, we deduced that
$$\dim H^{4g-6}(\cM_g;\Q) > 1.32^g + \text{constant},$$ 
disproving conjectures of Church, Farb, and Putman \cite[Conjecture~9]{ChurchFarbPutman14} and of Kontsevich \cite[Conjecture~7C]{Kontsevich93}, which would have implied that these cohomology groups vanish for all but finitely many $g$.

In this article, we expand on \cite{CGP1-JAMS}, in two main ways.
\begin{enumerate}
\item We introduce marked points. Given $g,n\ge 0$ such that $2g-2+n>0$, we study a space $\Dgn$ parametrizing stable tropical curves of genus $g$ with $n$ labeled, marked points (not necessarily distinct).  The case $n=0$ recovers the spaces $\Delta_g = \Delta_{g,0}$ studied in \cite{CGP1-JAMS}.
\item We are interested here in $\Dgn$ as a topological space, instead of studying only its rational homology.   For example, we prove that several large subspaces of $\Dgn$ are contractible, and determine the homotopy type of $\Delta_{1,n}$.
\end{enumerate}

As for (1), the introduction of marked points to the basic setup of \cite{CGP1-JAMS} poses no new technical obstacles.  In particular, we note that $\Dgn$ is the boundary complex of the Deligne-Mumford compactification $\ocM_{g,n}$ of $\cM_{g,n}$ by stable curves.
This identification implies that there is a natural isomorphism
\begin{equation}\label{eq:mgn-comparison}
\widetilde{H}_{k-1}(\Dgn;\Q)  \xrightarrow{\cong} \Gr_{6g-6+2n}^W H^{6g-6+2n-k} (\cM_{g,n}; \Q),
\end{equation}
identifying the reduced rational homology of $\Dgn$ with the top graded piece of the weight filtration on the cohomology of $\cM_{g,n}$. See \S\ref{sec:applications}.

As for (2), studying the combinatorial topology of $\Delta_{g,n}$ in more depth is a new contribution compared to \cite{CGP1-JAMS}, where we only needed the rational homology of $\Delta_g$ for our applications.  Note that $\Delta_{g,n}$ is a {\em symmetric $\Delta$-complex} (\S\ref{sec:cellular-chains}).  A new technical tool introduced in this article is a theory of collapses for symmetric $\Delta$-complexes, roughly analogous to discrete Morse theory for CW complexes (Proposition~\ref{prop:prop}).  We apply this tool to prove that several natural subcomplexes of $\Dgn$ are contractible.
Recall that an edge in a connected graph is called a \emph{bridge} if deleting it disconnects the graph.
\begin{restatable}{theorem}{contractible}  
\label{thm:contractible}
Assume $g > 0$ and $2g-2+n>0$.  Each of the following subcomplexes of $\Dgn$ is either empty or contractible.
\begin{enumerate}
\item \label{it:weight-contract} The subcomplex $\Dgn^{\rm w}$ of $\Dgn$ parametrizing tropical curves with at least one vertex of positive weight.
\item \label{it:loop-contract} The subcomplex $\Dgn^{\rm lw}$ of $\Dgn$ parametrizing tropical curves with loops or vertices of positive weight.
\item \label{it:rep-contract}
The subcomplex $\Dgn^{\rm rep}$ of $\Dgn$ parametrizing tropical curves in which at least two marked points coincide.
\item \label{it:br-contract} The closure $\Dgn^{\rm br}$ of the locus of tropical curves with bridges. 
\end{enumerate}
\end{restatable}

\noindent It is easy to classify when these loci are nonempty; see Remark~\ref{rem:when-nonempty}.  We refer to \cite[\S5.3]{ConantVogtmann03} and \cite{ConantGerlitsVogtmann05} for related results on contractibility of spaces of graphs with bridges.

\medskip

We use Theorem~\ref{thm:contractible} to deduce a number of consequences, which we outline below.

\medskip

{\bf The genus $1$ case.}  When $g=0$, the topology of the spaces $\Delta_{0,n}$ has long been understood; they are shellable simplicial complexes, homotopy equivalent to a wedge sum of $(n-2)!$ spheres of dimension $n-4$ \cite{Vogtmann90}.  Moreover, the character of $H_{n-4}(\Delta_{0,n};\Q)$ as an $S_n$-representation is computed in \cite{RobinsonWhitehouse96}.  Our results below give an analogous complete understanding when $g=1$.
Namely, the spaces $\Delta_{1,1}$ and $\Delta_{1,2}$ are easily seen to be contractible (Remark~\ref{rem:11-12-contractible}), and for $n\ge 3$, we have the following theorem.

\begin{restatable}{theorem}{boundary}  
\label{thm:boundary}  For $n\ge 3$, the space $\Dn$ is homotopy equivalent to a wedge sum of $(n-1)! / 2$ spheres of dimension $n-1$. The representation of $S_n$ on $H_{n-1}(\Delta_{1,n}; \Q)$ induced by permuting marked points is
\[
\mathrm{Ind}_{D_n,\phi}^{S_n} \, \mathrm{Res}^{S_n}_{D_n,\psi} \, \mathrm{sgn}.
\] 
\noindent Here, $\phi\colon D_n \rightarrow S_n$ is the dihedral group of order $2n$ acting on the vertices of an $n$-gon, $\psi\colon D_n \rightarrow S_n$ is the action of the dihedral group on the edges of the $n$-gon, and $\mathrm{sgn}$ denotes the sign representation of $S_n$.   
\end{restatable}

\noindent Note that the signs of these two permutation actions of the dihedral group are different when $n$ is even.  Reflecting a square across a diagonal, for instance, exchanges one pair of vertices and two pairs of edges. Moreover, calculating characters shows that these two representations of $D_4$ remain non-isomorphic after inducing along $\phi \colon  D_4 \to S_4$. 

Let us sketch how the expression in Theorem~\ref{thm:boundary} arises; the complete proof is given in \S\ref{sec:proofs}.  The $(n-1)!/2$ spheres mentioned in the theorem are in bijection with the $(n-1)!/2$ left cosets of $D_n$ in $S_n$; each may be viewed as a way to place $n$ markings on the vertices of an unoriented $n$-cycle.  Choosing left coset representatives $\sigma_1,\ldots,\sigma_k$ where $k=(n-1)!/2$, for any $\pi \in S_n$ we have $\pi \sigma_i = \sigma_j \pi'$ for some $\pi' \in D_n$.  Then, writing $[\sigma_i]$ for the fundamental class of the corresponding sphere, it turns out that the $S_n$-action on top homology of $\Delta_{1,n}$ is described as $\pi \cdot [\sigma_i] = \pm [\sigma_j]$, 
where the sign depends on the sign of the permutation on the {\em edges} of the $n$-cycle induced by $\pi'$.     This is because the ordering of edges determines the orientation of the corresponding sphere.  This implies that the $S_n$-representation on $H_{n-1}(\Delta_{1,n}; \Q)$ is exactly $\mathrm{Ind}_{D_n,\phi}^{S_n} \, \mathrm{Res}^{S_n}_{D_n,\psi} \, \mathrm{sgn}.$

Combining~\eqref{eq:mgn-comparison} and Theorem~\ref{thm:boundary}, and noting $S_n$-equivariance, gives the following calculation for the top weight cohomology of $\cM_{1,n}$.

\begin{corollary}  \label{thm:genus1}
The top weight cohomology of $\cM_{1,n}$ is supported in degree $n$, with rank $(n-1)!/2$, for $n \geq 3$.  Moreover, the representation of $S_n$ on $\Gr_{2n}^W H^{n}(\cM_{1,n}; \Q)$ induced by permuting marked points is
\[
\mathrm{Ind}_{D_n,\phi}^{S_n} \, \mathrm{Res}^{S_n}_{D_n,\psi} \, \mathrm{sgn}.
\]
\end{corollary}

\noindent Corollary~\ref{thm:genus1} is consistent with the recently proven formula for the $S_n$-equivariant top-weight Euler characteristic of $\cM_{g,n}$ in \cite{cgp-chi}.  See Remarks \ref{rem:M1n-stuff}, \ref{rem:more-M1n-stuff}, and \ref{rem:even-more-M1n-stuff} for further discussion of Corollary~\ref{thm:genus1} and its context.

In the case $g\ge 2$, we no longer have a complete understanding of the homotopy type of $\Delta_{g,n}$. 
However, the contractibility results in Theorem~\ref{thm:contractible} enable computer calculations of $H_i(\Delta_{g,n};\Q)$ for small $(g,n)$, presented in the Appendix.  

Theorem~\ref{thm:contractible} can also be used to deduce a lower bound on connectivity of the spaces $\Delta_{g,n}$. We do not pursue this here, but refer to \cite[Theorem 1.3]{cgp-arxiv}.

\bigskip

{\bf Marked graph complexes.}  In \cite{CGP1-JAMS} we gave a cellular chain complex $C_*(X;\Q)$ associated to any symmetric $\Delta$-complex $X$, and showed that it computes the reduced rational homology of (the geometric realization of) $X$.  In the case $X=\Delta_{g,n}$, we are able to deduce that $C_*(\Delta_{g,n};\Q)$ is quasi-isomorphic to the {\em marked graph complex} $\Ggn$, using Theorem~\ref{thm:contractible}.  We shall define $\Ggn$ precisely in \S\ref{sec:graphs}.  Briefly, it is generated by isomorphism classes of connected, $n$-marked stable graphs $\Gamma$ together with the choice of one of the two possible orientations of $\R^{E(\Gamma)}$. Kontsevich's graph complex $G^{(g)}$ \cite{Kontsevich93, Kontsevich94} occurs as the special case $n = 0$. The markings that we consider are elsewhere called hairs, half-edges, or legs; one difference between $\Ggn$ and many of the the hairy graph complexes in the existing literature \cite{ConantKassabovVogtmann13, ConantKassabovVogtmann15, KhoroshkinWillwacherZivkovic16, TurchinWillwacher17} is that our markings are ordered rather than unordered. Hairy graphs with ordered markings do appear in the work of Tsopm\'en\'e and Turchin on string links \cite{TsopmeneTurchin18a}; they study the more general situation where each marking carries a label from an ordered set $\{1, \ldots, r \}$, as well as the special case where multiple markings may carry the same label \cite[\S2.2.1]{TsopmeneTurchin18b}.  Our $\Ggn$ agrees with the complex denoted $M_g(P^n_2)$ in \cite{TsopmeneTurchin18b}.

\begin{restatable}{theorem}{split}
\label{thm:split}
For $g \geq 1$ and $2g -2 + n \geq 2$, there is a natural surjection of chain complexes
\[
C_* (\Delta_{g,n};\Q)\rightarrow G^{(g,n)},
\]
decreasing degrees by $2g - 1$, inducing isomorphisms on homology $$\widetilde H_{k+2g-1}(\Dgn;\Q) \xrightarrow{\cong}H_k (G^{(g,n)}) $$ for all $k$.
\end{restatable}

\noindent  We recover \cite[Theorem~1.3]{CGP1-JAMS}, in the special case where $n = 0$.   For an analogous result with coefficients in a different local system, see \cite[Proposition~27]{ConantVogtmann03}.

Combining~\eqref{eq:mgn-comparison} and Theorem \ref{thm:split} gives the following.

\begin{corollary} \label{thm:hairy-top-weight}
There is a natural isomorphism 
\[
H_k(G^{(g,n)}) \xrightarrow{\sim} \Gr_{6g-6+2n}^W H^{4g-6+2n-k} (\cM_{g,n}; \Q),
\]
identifying marked graph homology with the top weight cohomology of $\cM_{g,n}$.
\end{corollary}

Corollary~\ref{thm:hairy-top-weight} allows for an interesting application from moduli spaces back to graph complexes: applying known vanishing results for $\cM_{g,n}$, we obtain the following theorem for marked graph homology.

\begin{restatable}{theorem}{hgvanishing} \label{thm:hgvanishing}
The marked graph homology $H_{k}(G^{(g,n)})$ vanishes for $k < \max \{ 0, n-2 \}$.  Equivalently, $\widetilde H_k(\Dgn;\Q)$ vanishes for $k < \max\{ 2g - 1, 2g - 3 + n \}$.
\end{restatable}

\noindent  
Theorem~\ref{thm:hgvanishing} generalizes a theorem of Willwacher for $n = 0$.  See \cite[Theorem~1.1]{Willwacher15} and \cite[Theorem~1.4]{CGP1-JAMS}.
\bigskip

{\bf A transfer homomorphism.} 
It may be deduced from \cite[Theorem~1]{TurchinWillwacher17} and Theorem~\ref{thm:split} above
that $\widetilde{H}_k(\Dg;\Q)$ can be identified with a summand of $\widetilde H_k(\Delta_{g,1};\Q)$.  In the notation of op.cit., this is essentially the special case $n = m = 0$ and $h=1$ (their $\mathrm{HGC}_{0,0}^{g,1}$ is then a cochain complex isomorphic to a shift of our $(G^{(g,1)})^\vee$, while their  $\mathrm{GC}^g_{S^0 H_1}$ is isomorphic to our $(G^{(g)})^\vee$).  In \S\ref{sec:transfer} we give a proof of this in our setup, including an explicit construction of the splitting on the level of cellular chains of the tropical moduli spaces.

\begin{restatable}{theorem}{transfer}
\label{thm:transfer}
For $g \geq 2$, there is a natural homomorphism of cellular chain complexes
\[
t \colon  C_*(\Dg;\Q) \rightarrow C_*(\Delta_{g,1};\Q)
\]
which descends to a homomorphism $G^{(g)}\rightarrow G^{(g,1)}$ and induces injections $\widetilde H_k(\Dg;\Q) \hookrightarrow \widetilde H_k(\Delta_{g,1};\Q)$ and $H_k(G^{(g)}) \hookrightarrow H_k(G^{(g,1)})$, for all $k$.
\end{restatable}

\noindent The homomorphism $t$ is obtained as a weighted sum over all possible vertex markings, and may thus be seen as analogous to a transfer map. 
The resulting injection on homology is particularly interesting because $\bigoplus_g \widetilde H_{2g-1}(\Dg;\Q)$ is large: its graded dual is isomorphic to the Grothendieck-Teichm\"{u}ller Lie algebra, as discussed in \cite{CGP1-JAMS}. Combining with~\eqref{eq:mgn-comparison}, we obtain the following.

\begin{restatable}{corollary}{mgoneexp}
\label{cor:mgoneexp}
We have
$$\dim \Gr^W_{6g-4} H^{4g-4}(\cM_{g,1};\Q) > \beta^g +\text{constant}$$
for any $\beta < \beta_0 \approx 1.32\ldots$, where $\beta_0$ is the real root of $t^3-t-1=0.$
\end{restatable}

Corollary~\ref{cor:mgoneexp} can also be deduced purely algebro-geometrically from the analogous result for $\cM_g$ proved in \cite{CGP1-JAMS}, without using the transfer homomorphism $t$.  See Remark~\ref{rem:mg1}.  The following result is deduced by an easy application of the topological Gysin sequence, see Corollary~\ref{cor:1.9-strengthened}.

\begin{corollary}\label{cor:zero-section-complement}
  Let $\mathrm{Mod}_{g,1}$ denote the mapping class group of a genus $g$ surface with one parametrized boundary component.  Then
  $$\dim H^{4g-3}(\mathrm{Mod}_{g,1};\Q) > \beta^g +\text{constant}$$
  for any $\beta < \beta_0 \approx 1.32\ldots$ as above.
\end{corollary}

\bigskip

\noindent \textbf{Acknowledgments.}  We are grateful to E.~Getzler, A.~Kupers, D.~Petersen, O.~Randal-Williams, O.~Tommasi, K.~Vogtmann, and J.~Wiltshire-Gordon for helpful conversations related to this work.  MC was supported by NSF DMS-1204278, NSF CAREER DMS-1844768, NSA H98230-16-1-0314, and a Sloan Research Fellowship.  SG was supported by the European Research Council (ERC) under the European Union's Horizon 2020 research and innovation programme (grant agreement No 682922), the EliteForsk Prize, and by the Danish National Research Foundation (DNRF151 and DNRF151). SP was supported by NSF DMS-1702428 and a Simons Fellowship.  He thanks UC Berkeley and MSRI for their hospitality and ideal working conditions.

\section{Marked graphs and moduli of tropical curves}  \label{sec:graphs}

In this section, we recall the construction of the topological
space $\Dgn$ as a moduli space for genus $g$ stable, $n$-marked tropical curves.  The construction in \cite[\S2]{CGP1-JAMS} is the special case $n = 0$.  Following the usual convention, we write $\Dg=\Delta_{g,0}$.

\subsection{Marked weighted graphs and tropical curves} \label{subsec:J-g-n}

All graphs in this paper are connected, with loops and parallel edges allowed.
Let $G$ be a finite graph, with vertex set $V(G)$ and edge set $E(G)$.  A {\em weight function} is an arbitrary function $w\col V(G)\rightarrow \Z_{\ge 0}$. The pair $(G,w)$ is called a {\em weighted graph}.   Its {\em genus} is
\begin{equation*}
  g(G,w)= b_1(G) + \sum_{v\in V(G)}\! w(v),
\end{equation*}
where $b_1(G) = |E(G)|-|V(G)|+1$ is the first Betti number of $G$.
The {\em core} of a weighted graph $(G,w)$ is the smallest connected
subgraph of $G$ that contains all cycles of $G$ and all vertices of
positive weight, if $g(G,w)>0$, or the empty subgraph if $g(G,w)=0$.

An {\em $n$-marking} on $G$ is a map
$m\col\{1,\ldots,n\}\rightarrow V(G)$.  In figures, we depict
the marking as a set of $n$ labeled half-edges  or legs attached to the
vertices of $G$.

An {\em $n$-marked weighted graph} is a triple $\Gmw=(G,m,w)$, where
$(G,w)$ is a weighted graph and $m$ is an $n$-marking. The {\em
  valence} of a vertex $v$ in a marked weighted graph, denoted
$\val(v)$, is the number of half-edges of $G$ incident to $v$ plus the
number of marked points at $v$.  In other words, a loop edge based at
$v$ counts twice towards $\val(v)$, once for each end, an ordinary
edge counts once, and a marked point counts once (as suggested by the interpretation of markings as half-edges). We say that $\Gmw $
is {\em stable} if for every $v\in V(G)$, $$2w(v) -2 + \val(v) > 0.$$
Equivalently, $\Gmw$ is stable if and only if every vertex of weight 0
has valence at least 3, and every vertex of weight 1 has valence at
least 1.

\subsection{The category $\J_{g,n}$}\label{sec:jgn} The
stable $n$-marked graphs of genus $g$ form the objects of a
category which we denote $\J_{g,n}$.  The morphisms in this
category are compositions of contractions of edges $G\rightarrow G/e$
and isomorphisms $G \rightarrow G'$.  For the sake of removing any
ambiguity about what that might mean, we
now give a precise definition.

Formally, a graph $G$ is a finite set $X(G) = V(G) \sqcup H(G)$
(of ``vertices'' and ``half-edges''), together with two functions
$s_G, r_G\colon X(G) \to X(G)$ satisfying $s_G^2 = \mathrm{id}$ and 
$r_G^2 = r_G$, such that
$$\{x \in X(G) : r_G(x) = x\} = \{x \in X(G) : s_G(x) = x\} =
V(G).$$ Informally: $s_G$ sends a half-edge to its other half, while
$r_G$ sends a half-edge to its incident vertex.  We let
$E(G) = H(G)/(x \sim s_G(x))$ be the set of edges.  The definitions of
$n$-marking, weights, genus, and stability are as stated in the previous subsection.

The objects of the category $\J_{g,n}$ are all connected stable $n$-marked graphs of genus $g$.
For an object $\Gmw = (G,m,w)$ we shall write $V(\Gmw)$ for $V(G)$
and similarly for $H(\Gmw)$, $E(\Gmw)$, $X(\Gmw)$, $s_\Gmw$ and
$r_\Gmw$.  Then a morphism $\Gmw \to \Gmw'$ is a function $f\colon
X(\Gmw) \to X(\Gmw')$ with the property that
\begin{equation*}
f \circ r_\Gmw = r_{\Gmw'} \circ f \text{ and }f \circ s_\Gmw =
s_{\Gmw'} \circ f,
\end{equation*}
and subject to the following three requirements:  

\begin{itemize}
\item Noting first that $f(V(\Gmw))\subset V(\Gmw')$, we require
  $f\circ m = m'$, where $m$ and $m'$ are the respective marking
  functions of $\Gmw$ and $\Gmw'$.
\item Each $e \in H(\Gmw')$ determines the subset
  $f^{-1}(e) \subset X(\Gmw)$ and we require that it consists of
  precisely one element (which will then automatically be in
  $H(\Gmw)$).
\item Each $v \in V(\Gmw')$ determines a subset
  $S_v = f^{-1}(v) \subset X(\Gmw)$ and
  $\mathbf{S}_v = (S_v,r\vert_{S_v}, s\vert_{S_v})$ is a graph; we
  require that it be connected and have
  $g(\mathbf{S}_v,w\vert_{\mathbf{S}_v}) = w(v)$.
\end{itemize}

\noindent Composition of morphisms $\Gmw \to \Gmw' \to \Gmw''$ in
$\J_{g,n}$ is given by the corresponding composition
$X(\Gmw) \to X(\Gmw') \to X(\Gmw'')$ in the category of sets.

Our definition of graphs and the morphisms between them is standard in
the study of moduli spaces of curves and agrees, in essence, with the
definitions in \cite[X.2]{ACG11} and \cite[\S3.2]{acp}, as well as
those in \cite{KontsevichManin94} and \cite{GetzlerKapranov98}.  

\begin{remark}
  We also note that any morphism $\Gmw \to \Gmw'$ can be alternatively
  described as an isomorphism following a finite sequence of
  \emph{edge collapses}: if $e \in E(\Gmw)$ there is a canonical morphism
  $\Gmw \to \Gmw/e$ where $\Gmw/e$ is the marked weighted graph
  obtained from $\Gmw$ by collapsing $e$ together with its two
  endpoints to a single vertex $[e] \in \Gmw/e$.  If $e$ is not a
  loop, the weight of $[e]$ is the sum of the weights of the endpoints
  of $e$ and if $e$ is a loop the weight of $[e]$ is one more than the
  old weight of the end-point of $e$.  If
  $S = \{e_1, \dots, e_k\} \subset E(\Gmw)$ there are iterated edge
  collapses $\Gmw \to \Gmw/e_1 \to (\Gmw/e_1)/e_2 \to \dots$ and any
  morphism $\Gmw \to \Gmw'$ can be written as such an iteration
  followed by an isomorphism from the resulting quotient of $\Gmw$ to
  $\Gmw'$.
\end{remark}

We shall say that $\Gmw$ and $\Gmw'$ have the same \emph{combinatorial
  type} if they are isomorphic in $\J_{g,n}$.  In fact there are only
finitely many isomorphism classes of objects in $\J_{g,n}$, since any
object has at most $6g-6+2n$ half-edges and $2g-2+n$ vertices and for
each possible set of vertices and half-edges there are finitely many
ways of gluing them to a graph, and finitely many possibilities for
the $n$-marking and weight function.  In order to get a \emph{small}
category $\J_{g,n}$ we shall tacitly pick one object in each
isomorphism class and pass to the full subcategory on those objects.
Hence $\J_{g,n}$ is a \emph{skeletal} category.  (However, we shall try to 
use language compatible with any choice of small
equivalent subcategory of $\J_{g,n}$.)  It is clear that all Hom sets in
$\J_{g,n}$ are finite, so $\J_{g,n}$ is in fact a finite category.

Replacing $\J_{g,n}$ by some choice of skeleton has the effect that if
$\mathbf{G}$ is an object of $\J_{g,n}$ and $e \in E(\mathbf{G})$ is an
edge, then the marked weighted graph $\Gmw/e$ is likely not equal to
an object of $\J_{g,n}$.  Given $\mathbf{G}$ and $e$, there is a unique
\emph{morphism} $q \colon  \mathbf{G} \to \mathbf{G}'$ in $\J_{g,n}$ factoring
through an isomorphism $\Gmw/e \to \Gmw'$.  As usual it is the pair
$(\mathbf{G}',q)$ which is unique and not the isomorphism
$\Gmw/e \cong \Gmw'$.  By an abuse of notation, we shall henceforth
write $\mathbf{G}/e \in \J_{g,n}$ for the codomain of this unique
morphism, and similarly $G/e$ for its underlying graph.

\begin{definition}\label{def:E-and-H-as-functors}
  Let us define functors
  \begin{equation*}
    H,E \colon  \J_{g,n}^\mathrm{op} \to (\text{Finite sets},
    \text{injections})
  \end{equation*}
  as follows.  On objects, $H(\Gmw) = H(G)$ is the set of half-edges
  of $\Gmw = (G,m,w)$ as defined above.  A morphism
  $f \colon  \Gmw \to \Gmw'$ determines an injective function
  $H(f) \colon  H(\Gmw') \to H(\Gmw)$, sending $e' \in H(\Gmw')$ to the
  unique element $e \in H(\Gmw)$ with $f(e) = e'$.  We shall
  write $f^{-1} = H(f) \colon H(\Gmw') \to H(\Gmw)$ for this map.
  This clearly preserves composition and identities, and hence defines
  a functor.  Similarly for $E(\Gmw) = H(\Gmw)/(x \sim s_G(x))$ and
  $E(f)$.
\end{definition}

\subsection{Moduli space of tropical curves}  

We now recall the construction of moduli spaces of stable tropical
curves, as the colimit of a diagram of cones parametrizing possible
lengths of edges for each fixed combinatorial type, following \cite{BrannettiMeloViviani11, Caporaso13, acp}.

Fix integers $g,n\ge 0$ with $2g-2+n > 0$. A \emph{length function} on
$\Gmw = (G,m,w) \in \J_{g,n}$ is an element
$\ell \in \R_{>0}^{E(\Gmw)}$, and we shall think geometrically of
$\ell(e)$ as the \emph{length} of the edge $e \in E(\Gmw)$.  An
$n$-marked genus $g$ \emph{tropical curve} is then a pair
$\Gamma = (\Gmw,\ell)$ with $\Gmw \in \J_{g,n}$ and
$\ell \in \R_{>0}^{E(\Gmw)}$. We shall say that $(\Gmw,\ell)$ is
\emph{isometric} to $(\Gmw',\ell')$ if there exists an isomorphism
$\phi \colon  \Gmw \to \Gmw'$ in $\J_{g,n}$ such that
$\ell' = \ell \circ \phi^{-1} \colon  E(\Gmw') \to \R_{>0}$.  The
\emph{volume} of $(\Gmw,\ell)$ is
$\sum_{e \in E(\Gmw)} \ell(e) \in \R_{> 0}$.

We can now describe the underlying set of the topological space
$\Delta_{g,n}$, which is the main object of study in this paper. It is
the set of isometry classes of $n$-marked genus $g$ tropical curves of
volume 1.  We proceed to describe its topology and further structure
as a closed subspace of the moduli space of tropical curves.

\begin{definition}\label{definition:mgn}
  Fix $g,n\ge 0$ with $2g-2+n > 0$.  For each object
  $\Gmw \in \J_{g,n}$ define the topological space
  \begin{equation*}
    \sigma(\Gmw)= \R_{\geq 0}^{E(\Gmw)}=\{\ell \colon  E(\Gmw) \to \R_{\geq
      0}\}.
  \end{equation*}
For a morphism $f\colon \Gmw \to \Gmw'$ define the continuous map
  $\sigma f \colon \sigma(\Gmw') \to \sigma(\Gmw)$ by
  $$(\sigma f)(\ell') = \ell\colon E(\Gmw) \to \R_{\geq 0},$$ where
  $\ell$ is given by
  
  \begin{equation*}
    \ell(e) =
    \begin{cases} 
      \ell'(e') & \text{ if $f$ sends $e$ to $e'\in E(\Gmw')$},\\
      0 & \text{ if $f$ collapses $e$ to a vertex}.
    \end{cases}
  \end{equation*}
  This defines a functor
  $\sigma \colon  \J_{g,n}^\mathrm{op} \to \mathrm{Spaces}$ and the
  topological space $\Mgn$ is defined to be the colimit of this
  functor.
\end{definition}

In other words, the topological space $\Mgn$ is obtained as
follows. For each morphism $f\colon \Gmw\to\Gmw'$, consider the map
$L_f\colon \sigma(\Gmw') \to \sigma(\Gmw)$ that sends
$\ell'\colon E(\Gmw')\to \R_{>0}$ to the length function
$\ell\colon E(\Gmw)\to \R_{>0}$ obtained from $\ell'$ by extending it
to be 0 on all edges of $\Gmw$ that are collapsed by $f$.  So $L_f$
linearly identifies $\sigma(\Gmw')$ with some face of $\sigma(\Gmw)$,
possibly $\sigma(\Gmw)$ itself.  Then
\begin{equation*}
  \Mgn = \left(\coprod \sigma(\Gmw)\right) \Big / \{\ell' \sim L_f(\ell')\},
\end{equation*}
where the equivalence runs over all morphisms $f\colon \Gmw\to\Gmw'$
and all $\ell'\in \sigma(\Gmw')$.

As we shall explain in more detail in later sections, $\Mgn$ naturally
comes with more structure than just a topological space: $\Mgn$ is a {\em generalized cone complex}, as defined in \cite[\S2]{acp}, and associated to a \emph{symmetric $\Delta$-complex} in the sense of \cite{CGP1-JAMS}.  This formalizes the observation that $\Mgn$ is glued
out of the cones $\sigma(\Gmw)$.

The \emph{volume} defines a function
$v \colon  \sigma(\Gmw) \to \R_{\geq 0}$, given explicitly as
$v(\ell) = \sum_{e \in E(\Gmw)} \ell(e)$, and for any morphism
$\Gmw \to \Gmw'$ in $\J_{g,n}$ the induced map
$\sigma(\Gmw) \to \sigma(\Gmw')$ preserves volume.  Hence there is an
induced map $v \colon  \Mgn \to \R_{\geq 0}$, and there is a unique element
in $\Mgn$ with volume 0 which we shall denote $\bullet_{g,n}$.  The
underlying graph of $\bullet_{g,n}$ consists of a single vertex with
weight $g$ that carries all $n$ marked points.

\begin{definition}\label{def:dgn}
  We let $\Dgn$ be the subspace of $\Mgn$ parametrizing curves of
  volume 1, i.e.,\ the inverse image of $1 \in \R$ under
  $v \colon  \Mgn \to \R_{\ge 0}$.
\end{definition}

\noindent Thus $\Dgn$ is homeomorphic to the link of $\Mgn$ at the
cone point $\bullet_{g,n}$.

\subsection{The marked graph complex.}  \label{ssec:Ggn}

Fix integers $g,n\ge 0$ with $2g-2+n>0$.  The {\em marked graph complex} $G^{(g,n)}$ is a chain complex of rational vector spaces.  As a graded vector space, it has generators $[\Gamma,\omega, m]$ for each connected graph $\Gamma$ of genus $g$ (Euler characteristic $1-g$) with or without loops, equipped with a total order $\omega$ on its set of edges and a marking $m \colon \{ 1, \ldots, n \} \rightarrow V(G)$, such that at each vertex $v$, the number of half-edges incident to $v$ plus $|m^{-1}(v)|$ is at least 3.  These generators are subject to the relations $$[\Gamma,\omega, m] = \mathrm{sgn}(\sigma) [\Gamma',\omega', m']$$ if there exists an isomorphism of graphs $\Gamma \cong \Gamma'$ that identifies $m$ with $m'$, and under which the edge orderings $\omega$ and $\omega'$ are related by the permutation $\sigma$.  In particular this forces $[\Gamma,\omega, m] = 0$ when $\Gamma$ admits an automorphism that fixes the markings and induces an odd permutation on the edges.  A genus $g$ graph $\Gamma$ with $v$ vertices and $e$ edges is declared to be in homological degree $v-(g+1) = e-2g$.  When $n = 0$, this convention agrees with \cite{Willwacher15}; it is shifted by $g+1$ compared to \cite{Kontsevich93}.

For example, $G^{(1,1)}$ is $1$-dimensional, supported in degree $0$, with a generator corresponding to a single loop based at a vertex supporting the marking.  On the other hand, $G^{(2,0)} = 0$, since all generators of $G^{(2,0)}$ are subject to the relation $[\Gamma,\omega,m] = -[\Gamma,\omega,m]$.

The differential on $\Ggn$ is defined as follows: given $[\Gamma,\omega,m]\ne 0$,
\begin{equation}\label{eq:the-differential}
\partial[\Gamma,\omega,m]= \sum_{i=0}^N (-1)^i [\Gamma/e_i, \omega|_{E(\Gamma)\smallsetminus\{e_i\}}, \pi_i\circ m],
\end{equation} where
$\omega = (e_0<e_1<\cdots<e_N)$ is the total ordering on the edge set $E(\Gamma)$ of $\Gamma$, the graph $\Gamma/e_i$ is the result of collapsing $e_i$ to a point, $\pi_i\colon V(\Gamma)\to V(\Gamma/e_i)$ is the resulting surjection of vertex sets, and $\omega|_{E(\Gamma/e_i)}$ is the induced ordering on the subset $E(\Gamma/e_i) = E(\Gamma)\smallsetminus \{e_i\}$.  
If $e_i$ is a loop, we interpret the corresponding term  in~\eqref{eq:the-differential} as zero.

\begin{remark}
We may equivalently define $G^{(g,n)}$ as follows: take the graded vector space
$$\bigoplus_{(\Gamma,m)} {\mathsf \Lambda}^{|E(\Gamma)|} \Q^{E(\Gamma)},$$
where $ {\mathsf \Lambda}^{|E(\Gamma)|} \Q^{E(\Gamma)}$ appears in homological degree $|E(\Gamma)|-2g$, and impose the following relations.  For any isomorphism $\phi\col (\Gamma,m) \to (\Gamma',m')$, let $$\phi_*\col {\mathsf \Lambda}^{|E(\Gamma)|} \Q^{E(\Gamma)} \to {\mathsf \Lambda}^{|E(\Gamma')|} \Q^{E(\Gamma')}$$
denote the induced isomorphism of $1$-dimensional vector spaces. Then we set $w = \phi_*w$ for each $w\in {\mathsf \Lambda}^{|E(\Gamma)|} \Q^{E(\Gamma)}$.  Viewing nonzero elements of ${\mathsf \Lambda}^{|E(\Gamma)|} \Q^{E(\Gamma)}$ as orientations on $\R^{|E(\Gamma)|}$, this description accords with the rough definition of $G^{(g,n)}$ in terms of graphs and orientations given in the introduction.
\end{remark}

\section{Symmetric $\Delta$-complexes and relative cellular homology}
\label{sec:cellular-chains}

We briefly recall the notion of symmetric $\Delta$-complexes and explain how $\Dgn$ is naturally interpreted as an object in this category.  We then recall the cellular homology of symmetric $\Delta$-complexes developed in \cite[\S3]{CGP1-JAMS}, and extend this to a cellular theory of relative homology for pairs.  This relative cellular homology of pairs will be applied in \S\ref{sec:proofs} to prove that there is a surjection $C_*(\Delta_{g,n};\Q) \rightarrow G^{(g,n)}$ that induces isomorphisms in homology, as stated in Theorem~\ref{thm:split}.

\subsection{The symmetric $\Delta$-complex structure on $\Dgn$}  Let $I$ denote the category whose objects are the sets $[p] = \{0, \ldots, p \}$ for nonnegative integers $p$, together with $[-1] := \emptyset$, and whose morphisms are injections of sets.  
\begin{definition}
A {\em symmetric $\Delta$-complex} is a presheaf on $I$, i.e., a functor from $I^{\mathrm{op}}$ to $\mathsf{Sets}$.  
\end{definition}
As in \cite[\S3]{CGP1-JAMS} we write $X_p = X([p])$, except when that would create double subscripts.
The {\em geometric realization} of $X$ is the topological space
\begin{equation}\label{eq:realization}
|X| = \bigg( \coprod_{p = 0}^\infty X_p \times \Delta^p \bigg) / \sim,
\end{equation}
where $\Delta^p$ is the standard $p$-simplex and $\sim$ is the equivalence relation that is generated as follows.  For each $x \in X_p$ and each injection $\theta  \colon  [q] \hookrightarrow [p]$, the simplex $\{\theta^*(x)\} \times \Delta^q$ is identified with a face of $\{x\} \times \Delta^{p}$ via the linear map that takes the vertex $(\theta^*(x),e_i)$ to $(x,e_{\theta(i)} )$.  The cone $CX$ is defined similarly, but with $\R^{p+1}_{\geq 0}$ in place of $\Delta^p$.
Note that our symmetric $\Delta$-complexes include the data of an {\em augmentation}, that is, the locally constant map $|X| \rightarrow X_{-1}$ induced by the unique inclusion $[-1] \hookrightarrow [p]$.

\begin{example}\label{ex:Dgn}
  Most importantly for our purposes, 
  $\Dgn$ is naturally identified with the geometric realization of a symmetric
  $\Delta$-complex (and $\Mgn$ is the associated cone), as we now explain.  

Consider the following functor $X=X_{g,n}\colon I^\mathrm{op} \to \mathsf{Sets}$. The elements of $X_p$ are equivalence classes of pairs $(\Gmw,\tau)$ where 
$\Gmw\in \J_{g,n}$ and $\tau\colon E(\Gmw)\to [p]$ is a bijective edge-labeling.  
  Two edge-labelings are considered equivalent if they are in the same
  orbit under the evident action of $\Aut(\Gmw)$.  Here $\Gmw$ ranges
  over all objects in $\J_{g,n}$ with exactly $p+1$ edges.  (Recall
  from \S\ref{sec:jgn} that we have tacitly picked one element
  in each isomorphism class in $\J_{g,n}$.)

  Next, for each injective map $\iota\colon [p']\to[p]$, define the
  following map $X(\iota) \colon X_p \to X_{p'}$; given an element
  of $X_p$ represented by $(\Gmw,\tau\colon E(\Gmw)\to [p])$,
  contract the edges of $\Gmw$ whose labels are not in
  $\iota([p'])\subset [p]$, then relabel the remaining edges with
  labels $[p']$ as prescribed by the map $\iota$.  The result is a
  $[p']$-edge-labeling of some new object $\Gmw'$, and we set
  $X(\iota)(\Gmw)$ to be the corresponding element of $X_{p'}$.

The geometric realization of $X$ is naturally identified with $\Dgn$, as follows.  Recall that a point in the relative interior of $\Delta^p$ is expressed as $\sum_{i=0}^p a_i e_i$, where $a_i >0$ and $\sum a_i = 1$.  Then an element $x \in X_p$ corresponds to a graph in $\J_{g,n}$ together with a labeling of its $p+1$ edges, and a point $(x,\sum a_i e_i)$ corresponds to the isomorphism class of stable tropical curve with underlying graph $x$ in which the edge labeled $i$ has length $a_i$.  By abuse of notation, we will use $\Dgn$ to refer to this symmetric $\Delta$-complex, as well as its geometric realization.
\end{example}

\begin{remark}
  We note that the cellular complexes of $\Delta_{g,n}$ have natural
  interpretations in the language of modular operads, developed by
  Getzler and Kapranov to capture the intricate combinatorial
  structure underlying relations between the $S_n$-equivariant
  cohomology of $\cM_{g,n}$, for all $g$ and $n$, and that of
  $\ocM_{g',n'}$, for all $g'$ and $n'$ \cite{GetzlerKapranov98}.  In
  particular, the cellular cochain complexes $C^*(\Delta_{g,n})$, with
  their $S_n$-actions, 
  agree with the Feynman transform of the modular operad $\mathrm{ModCom}$, assigning the vector space $\Q$, with trivial $S_n$-action, to each $(g,n)$  with  $2g-2+n >0$, and assigns 0 otherwise.  There is a quotient map $\mathrm{ModCom} \to \mathrm{Com}$, which is an isomorphism for $g = 0$ and where the commutative operad $\mathrm{Com}$ is zero for $g > 0$.  The Feynman transform of this quotient map is, up to a regrading, the map appearing in Theorem~\ref{thm:split}.  Since the Feynman transform is homotopy invariant and has an inverse up to quasi-isomorphism, it cannot turn the non-quasi-isomorphism $\mathrm{ModCom} \to \mathrm{Com}$ into a quasi-isomorphism, and therefore the map in Theorem~\ref{thm:split} cannot be a quasi-isomorphism for all $(g,n)$.  In fact it is not in the exceptional cases $g = 0$ and $(g,n) = (1,1)$, where $\weight = \emptyset$.

  See also \cite{AWZ20}, especially \S6.2 and \S3.2.1.
  \end{remark}

\subsection{Relative homology}

We now present a relative cellular homology for pairs of symmetric $\Delta$-complexes.  This is a natural extension of the cellular homology theory for symmetric $\Delta$-complexes developed in \cite[\S3]{CGP1-JAMS}, which we briefly recall.  

Let $X  \colon  I^\mathrm{op} \rightarrow \mathrm{Sets}$ be a symmetric $\Delta$-complex.
The group of cellular $p$-chains $C_p(X)$ is defined to be the co-invariants
\[ 
C_p(X) = (\Q^{\mathrm{sign}} \otimes_\Q \Q {X_p})_{S_{p+1}},
\]
where $\Q X_p$ denotes the $\Q$-vector space with basis $X_p$ on which $S_{p+1} = I^{\mathrm{op}}([p],[p])$ acts by permuting the basis vectors.  This comes with a natural differential $\partial$ induced by
$\sum(-1)^i (d_i)_* \colon  \Q X_p \rightarrow \Q X_{p-1}$, and the homology of $ C_* (X)$ is identified with the rational homology $\widetilde H_*(|X|, \Q)$, reduced with respect to the augmentation $|X| \rightarrow X_{-1}$, cf. \cite[Proposition~3.8]{CGP1-JAMS}.

\label{sec:relative-homology}

This rational cellular chain complex has the following natural analogue for \emph{pairs} of
symmetric $\Delta$-complexes: there is a relative cellular chain
complex, computing the relative (rational) homology of the geometric
realizations.  Let us start by discussing what ``subcomplex'' should
mean in this context.

\begin{lemma}
  Let $X, Y \colon  I^\mathrm{op} \to \mathrm{Sets}$ be symmetric
  $\Delta$-complexes, and let $f \colon  X \to Y$ be a map (i.e.,\ a natural
  transformation of functors).  If $f_p \colon  X_p \to Y_p$ is
  injective for all $p \geq 0$, then $|f| \colon  |X| \to |Y|$ is also
  injective.

  If $f_p$ is injective for all $p \geq -1$, then $Cf: CX \to CY$ is also injective, where $CX$ and $CY$ are the cones over $X$ and $Y$.\end{lemma}
\begin{proof}[Proof sketch]
  Let us temporarily write $H^X(x) < S_{p+1}$ for the stabilizer of a
  simplex $x \in X_p$, and similarly $H^Y(f(x))$ for the stabilizer
  of $f(x) \in Y_p$.  The maps $f_p \colon  X_p \to Y_p$ are
  equivariant for the $S_{p+1}$ action, and injectivity of $f_p$
  implies that $H^X(x) = H^Y(f(x))$ and the induced map of
  orbit sets $X_p/S_{p+1} \to Y_p/S_{p+1}$ is
  injective.

  As a set, $|X|$ is the disjoint union of
  $(\Delta^p\smallsetminus \partial \Delta^p)/H^X(x)$ over all $p$ and one
  $x \in X_p$ in each $S_{p+1}$-orbit, and similarly for $|Y|$.  At the level
  of sets, the induced map $|f| \colon  |X| \to |Y|$ then restricts to a
  bijection from each subset
  $(\Delta^p\smallsetminus \partial \Delta^p)/H^X(x) \subset |X|$ onto the
  corresponding subset of $|Y|$, and these subsets of $|Y|$ are
  disjoint.

  The statement about $CX \to CY$ is proved in a similar way.
\end{proof}
Conversely, it is not hard to see that the geometric map $|f| \colon  |X| \to |Y|$ is injective
\emph{only} when $f_p \colon  X_p \to Y_p$ is injective for all $p \geq 0$.
In this situation, in fact $|f| \colon  |X| \to |Y|$ is always a
homeomorphism onto its image.  This is obvious in the case of finite
complexes where both spaces are compact, which is the only case needed
in this paper, and in general is proved in the same way as for CW
complexes (the CW topology on a subcomplex agrees with the subspace
topology).

\begin{definition}
  A \emph{subcomplex} $X \subset Y$ of a symmetric $\Delta$-complex
  is a subfunctor $X$ of $Y\col I^{\mathrm{op}} \to \mathrm{Sets}$, in which for each $p$, $X_p$ is a subset of  $Y_p$, with the subfunctor being given by the canonical inclusions $X_p \to Y_p$.
  The inclusion $\iota \colon  X \to
  Y$ then induces an injection $|\iota|  \colon  |X| \to |Y|$, which we shall
  use to identify $|X|$ with its image $|X| \subset |Y|$.
\end{definition}
\noindent In particular, we emphasize that for each injection $\iota\col [p']\to [p]$ the map $X(\iota)\col X_p\to X_{p'}$ is a restriction of $Y(\iota)\col Y_p\to Y_{p'}$.  

If $X\subset Y$ is a subcomplex of a symmetric $\Delta$-complex, we obtain a map $\iota_* \colon   C_*(X) \to
 C_*(Y)$ of cellular chain complexes which is injective in each degree,
and we define a relative cellular chain complex $ C_*(Y,X)$ by the
short exact sequences
\begin{equation}\label{eq:1}
  0 \to  C_p(X) \xrightarrow{\iota_p}  C_p(Y) \to
   C_p(Y,X) \to 0
\end{equation}
for all $p \geq -1$.  Similarly for cochains and with coefficients in
a $\Q$-vector space $A$.

\begin{proposition}
  Let $Y$ be a symmetric $\Delta$-complex and $X \subset Y$ a subcomplex.  Let $\iota: X \to Y$ be the inclusion, and let $C\iota: CX \to CY$ be the induced maps of cones over $X$ and $|\iota|: |X| \to |Y|$ be the restriction of $C\iota$.  Then there is a natural isomorphism
  \begin{equation*}
    H_{p+1}((CY/CX),(|Y|/|X|);\Q) \cong H_p( C_*(Y,X)).
  \end{equation*}
  In particular, if $f_{-1}: X_{-1} \to Y_{-1}$ is a bijection, we get a natural isomorphism
  \begin{equation*}
    H_p(C_*(Y,X)) \cong H_p(|Y|,|X|;\Q).
  \end{equation*}
  Similarly for cohomology.  A similar result also holds with coefficients in an arbitrary abelian group $A$, provided that $S_{p+1}$ acts freely on $X_p$ and $Y_p$ for
all $p$.
\end{proposition}
\begin{proof}[Proof sketch]
  We prove the statement about homology.  The short exact sequence~(\ref{eq:1}) induces a long exact sequence in homology, which maps to the
long exact sequence in singular homology of the geometric
realizations.  For each $p$, the first, second, fourth, and fifth vertical arrows

$$\xymatrix{H_{p}(C_*(X))\ar[r]\ar[d]^{\cong} & H_{p}(C_*(Y)) \ar[r]\ar[d]^{\cong} &H_{p}(C_*(Y,X))\ar[r]\ar[d] &H_{p-1}(C_*(X)) \ar[r]\ar[d]^{\cong} &H_{p-1}(C_*(Y)) \ar[d]^{\cong}\\
H_{p}(|X|;\Q) \ar[r] & H_{p}(|Y|;\Q) \ar[r] & H_{p}(|Y|,|X|;\Q) \ar[r] & H_{p-1}(|X|;\Q) \ar[r]  & H_{p-1}(|Y|;\Q)}$$
are shown to be isomorphisms in \cite[Proposition~3.8]{CGP1-JAMS}, so the middle vertical arrow is also an isomorphism. 
\end{proof}

\section{A contractibility criterion} \label{sec:contractibility}

In this section, we develop a general framework for contracting subcomplexes of a symmetric $\Delta$-complex, loosely in the spirit of discrete Morse theory.  We then apply this technique to prove Theorem~\ref{thm:contractible}, showing that three natural subcomplexes of $\Dgn$ are contractible.
See Figure~\ref{fig:master} for a running illustration of the various definitions that follow.

\subsection{A contractibility criterion}\label{sec:machine}

Let $X$ be a symmetric $\Delta$-complex, i.e.,~a functor $I^{op}\rightarrow\mathrm{Sets}$.  For each injective map $\theta\colon[q]\rightarrow [p]$ we write $$\theta^*\colon X_p\rightarrow X_q,$$ for the induced map on simplices, where $X_i$ denotes the set of $i$-simplices $X([i])$.  

For $\sigma\in X_p$ and $\theta\colon [q]\hookrightarrow[p]$, we say that $\tau = \theta^*\sigma \in X_q$ is a {\em face} of $\sigma$, with face map $\theta$, and we write $\tau \precsim \sigma$. Thus $\precsim$ is a reflexive and transitive relation. It descends to a partial order $\preceq$ on the set $\coprod X_p/S_{p+1}$ of symmetric orbits of simplices.  We write $[\sigma]$ for the $S_{p+1}$-orbit of a $p$-simplex $\sigma$.  

\begin{definition}
A {\em property} on $X$ is a subset of the vertices $P\subset X_0$.
\end{definition}
\noindent One could call this a ``vertex property,'' but we avoid that terminology since our motivating example is $X = \Dgn$, when the vertices of $X$ are $1$-edge graphs. In this situation, we are interested in properties of {\em edges} of graphs in $\J_{g,n}$ that are preserved by automorphisms and uncontractions, such as the property of being a bridge.  See Example~\ref{ex:properties-1}.

\medskip
\begin{figure}
\begin{tikzpicture}
\tikzstyle{vert}=[circle,thick,draw=gray!60,fill=gray!60,inner sep=1pt]
\tikzstyle{blue}=[circle,thick,draw=blue!60,fill=blue!60,inner sep=1pt]
\begin{scope}[scale=3]
\node (a)[vert] at (-.5,.86) {};
\node (b)[vert] at (0.5, .86) {};
\node (c)[vert] at (-1, 0) {};
\node (d)[vert] at (0, 0) {};
\node (e)[vert] at (1, 0) {};
\draw [gray!50, ultra thick, fill=gray!20] (a.center) -- (c.center) -- (d.center) -- (a.center);
\draw [gray!50, ultra thick, fill=gray!20] (a.center) -- (b.center) -- (d.center) -- (a.center);
\draw [gray!50, ultra thick, fill=gray!20] (e.center) -- (b.center) -- (d.center) -- (e.center);
\node [below] at (0, 0) {$\scriptstyle{P}$};
\node [below] at (1, 0) {$\scriptstyle{P}$};
\node [gray] at (0,.95) {$\scriptstyle{P^*\smallsetminus P}$};
\node [gray] at (-.9,.5) {$\scriptstyle{P^*\smallsetminus P}$};
\node [gray] at (-1.2,0) {$\scriptstyle{P^*\smallsetminus P}$};
\node [gray] at (-.5,.95) {$\scriptstyle{P^*\smallsetminus P}$};
\node [gray] at (.5,.95) {$\scriptstyle{P^*\smallsetminus P}$};
\draw [blue!50, ultra thick, fill=blue!50] (d.center)--(e.center);
\node [blue] at (0, 0) {};
\node [blue] at (1, 0) {};
\node [black] at (.5, .76) {$\scriptstyle{v}$};
\node [black] at (.5, .36) {$\scriptstyle{T}$};
\end{scope}
\end{tikzpicture}
\caption{A symmetric $\Delta$-complex $X$ with two $P$-vertices.  Here $X_P = X$, and $v$ is a co-$P$-face of $T$.}
\label{fig:master}
\end{figure}
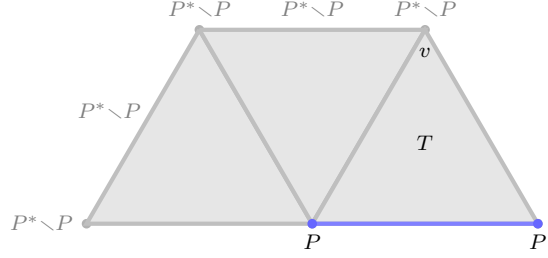

Let $P\subset X_0$ be a property, and let $\sigma\in X_p$.  For $i=0,\ldots, p$, we understand the $i^{\text{th}}$ vertex of $\sigma$ to be $v_i = \iota^*(\sigma)$, where $\iota\colon [0]\rightarrow[p]$ sends $0$ to $i$.  We write
 $$P(\sigma) = \{i \in [p] : v_i \text{ is in }P\}$$
 for the set labeling vertices of $\sigma$ that are in $P$, and call these {\em $P$-vertices} of $\sigma$.  Similarly, we write
 $$P^c(\sigma) = \{i \in [p] : v_i \text{ is not in }P\}$$
 for the complementary set, and call these the {\em non-$P$-vertices} of $\sigma$.

We write $\Simp(X) = \coprod_{p\ge 0} X_p$ for the set of all simplices of $X$, and define $$P(X) = \{\sigma \in \Simp(X) : P(\sigma)\ne \emptyset\}.$$ We call the elements of $P(X)$ the {\em $P$-simplices} of $X$; they are the simplices with at least one $P$-vertex.   If $P^c(\sigma)=\emptyset$ then we say $\sigma$ is a {\em strictly $P$-simplex}.

Any collection of simplices of $X$ naturally generates a subcomplex whose simplices are all faces of simplices in the collection.  We write $X_P$ for the subcomplex of $X$ generated by $P(X)$.  Let $P^*(X)$ denote the set of simplices of $X_P$.  In other words,
$$P^*(X) = \Simp(X_P) = \{\tau\in \Simp(X) : \tau \precsim \sigma \text{ for some } \sigma\in P(X)\}.$$  
The set $\Simp(X) \smallsetminus P(X)$ is also the set of simplices in a subcomplex of $X$, as is $P^*(X) \smallsetminus P(X)$.

\begin{example}\label{ex:master}
Figure~\ref{fig:master} shows a symmetric $\Delta$-complex $X$ with five 0-simplices, $7$ symmetric orbits of $1$-simplices and $3$ symmetric orbits of $2$-simplices; we have chosen to illustrate an example where $S_{p+1}$ acts freely on the $p$-simplices for each $p$.  

There are two vertices in $P$.  
The subcomplex $X_P$ is then all of $X$, since the maximal simplices all have at least one $P$-vertex.  In this example, we therefore have $P^*(X) = \Simp(X)$.  The only simplices {\em not} in $P(X)$ are marked $P^*\smallsetminus P$ in the figure. The strictly $P$-simplices are drawn in blue.
\end{example}
 
 \begin{example}\label{ex:properties-1}
We pause to explain how these definitions apply to $X=\Delta_{g,n}$.    For any $g,n\ge0$ with $2g-2+n>0$,
recall that the $p$-simplices in $\Delta_{g,n}$ are pairs $(\Gmw,t)$ where $\Gmw\in \J_{g,n}$ and $t\colon E(\Gmw)\rightarrow [p]$ is a bijection.  There is a bijection from $\coprod_{p\ge -1} \Delta_{g,n}([p])/S_{p+1}$ to $\J_{g,n}$, sending $[(\Gmw,t)]$ to $\Gmw$. 

The vertices $X_0$ are simply one-edge graphs $\Gmw\in \J_{g,n}$, since any such graph has a unique edge labeling. There is one such graph which is a loop; the others are bridges.  For each $g'$ satisfying $0 \leq g' \leq g$ and each subset $A \subset [n]$ with $2g'-1+|A|>0$ and $2(g-g')-1+(n-|A|)>0$, there is a unique one-edge graph in $\J_{g,n}$ with vertices $v_1$ and $v_2$, such that $w(v_1) = g'$ and $m^{-1}(v_1) = A$.
We write $\bfB(g',A) \in \Delta_{g,n}([0])$ for the corresponding vertex.

Note that $\bfB(g',A)=\bfB(g - g',[n] \smallsetminus A)$.  We define the property
$$P_{g',n'} = \{\bfB(g',A) : |A| = n'\} \subset \Delta_{g,n}([0]).$$  A simplex $\sigma = (\Gmw,t)\in \Delta_{g,n}$ is a $P_{g',n'}$-simplex if and only if $\Gmw$ has a {\em $(g',n')$-bridge}, i.e.,~a bridge separating subgraphs of types $(g',n')$ and $(g\!-\!g',n\!-\!n')$ respectively.  Similarly, 
 $(\Gmw',t')\in P_{g',n'}^*(X)$ if $\Gmw'$ admits a morphism in $\J_{g,n}$ from some $\Gmw$ with a $(g',n')$-bridge.  
\end{example}

We return to the general case, where $X$ is a symmetric $\Delta$-complex and $P\subset X_0$ is a property, and define a co-$P$ face as a face such that all complementary vertices lie in $P$.  

\begin{definition}
Given $\sigma\in X_p$ and $\theta\colon [q]\hookrightarrow [p]$, we say that $\theta$ is a {\em co-$P$ face map} if $[p]\smallsetminus \operatorname{im} \theta \subset P(\sigma)$.  In this case, we say that $\tau = \theta^*(\sigma)$ is a {\em co-$P$ face} of $\sigma$.  
\end{definition}

\noindent We write $\tau \precsim_P \sigma$ if $\tau$ is a co-$P$ face of $\sigma$.  Then $\precsim_P$ is a reflexive, transitive relation, and it induces a partial order $\preceq_P$ on $\coprod_{p\ge 0} X_p/S_{p+1}$, where $[\tau]\preceq_P [\sigma]$ if $\tau\precsim_P\sigma$.

\begin{example}
In Figure~\ref{fig:master}, the $0$-simplex $v$ is a co-$P$-face of $T$.  

  In our main example $X=\Delta_{g,n}$, for any property $P\subset X_0$, let us say that an edge $e\in E(\Gmw)$ is a {\em $P$-edge} if the graph obtained from $\Gmw$ by collapsing each element of $E(\Gmw)-\{e\}$ is in $P$.
  A {\em $P$-contraction} is a contraction of $\Gmw$ by a subset, possibly empty, of $P$-edges.
Then a face $(\Gmw',t')$ of $(\Gmw,t)$ is a co-$P$ face if and only if $\Gmw'$ is isomorphic to a $P$-contraction of $\Gmw$.
\end{example}

The \emph{automorphisms} of a simplex $\sigma \in X_p$, denoted $\Aut(\sigma)$, are the bijections $\psi\colon [p]\rightarrow[p]$ such that $\psi^*\sigma = \sigma.$  The natural map $\{\sigma\} \times \Delta^p \rightarrow |X|$ factors through $\Delta^p/ \Aut(\sigma)$.

A face $\tau \precsim \sigma$ is {\em canonical} (meaning canonical up to automorphisms) if, for any two injections $\theta_1$ and $\theta_2$ from $[q]$ to $[p]$ such that $\theta_i^*(\sigma) = \tau$, there exists $\psi \in \Aut(\sigma)$ such that $\theta_1 = \psi \theta_2$.  Note that the property of $\tau \precsim \sigma$ being canonical depends only on the respective $S_{q+1}$ and $S_{p+1}$ orbits; we will say that $[\tau] \preceq [\sigma]$ is canonical if $\tau \precsim \sigma$ is so.

\begin{remark}
  In \cite[\S3.4]{CGP1-JAMS} we defined a category $J_X$ with object set $\coprod_{p \geq -1} X_p$.  Automorphisms of $\sigma$ in this category agree with automorphisms in the above sense, and the relation $\tau \precsim \sigma$ holds if and only if there exists a morphism $\sigma \to \tau$ in $J_X$.  In op.\ cit.\ we also defined a $\Delta$-complex $\mathrm{sd}(X)$ called the subdivision of $X$, and a canonical homeomorphism $|\mathrm{sd}(X)| \cong |X|$.  Geometrically,  $[\sigma]$ and $[\tau]$ are then 0-simplices of $\mathrm{sd}(X)$, they are related by $\preceq$ if and only if there exists a 1-simplex connecting them.  The relation is canonical if and only if there is precisely one 1-simplex in $\mathrm{sd}(X)$ between them.
\end{remark}

\begin{example}
For any subgroup $G < S_{p+1}$, the quotient $\Delta^p / G$ carries a natural structure of symmetric $\Delta$-complex in which every face is canonical.  For an example of a face inclusion that is not canonical, consider the $\Delta$-complex consisting of a loop formed by one vertex and one edge (viewed as a symmetric $\Delta$-complex with one $0$-simplex and two $1$-simplices, cf. \cite[\S3]{CGP1-JAMS}). The automorphism groups of all simplices are trivial, so neither of the vertex-edge inclusions is canonical.  For an example of noncanonical face inclusions in $\Dgn$, see Example~\ref{ex:nonexample-Dgn}.
\end{example}

The main technical result of this section, Proposition~\ref{prop:prop}, involves canonical co-$P$-maximal faces and co-$P$-saturation, defined as follows. See Example~\ref{ex:canon-sat} below.

\begin{definition} \label{def:the-canon}
Let $Z\subset \coprod X_p/S_{p+1}$, and let $P$ be any property. We say that $Z$ admits {\em canonical co-$P$ maximal faces} if, 
for every $[\tau] \in Z$, the poset of those $[\sigma] \in Z$ such that $[\tau] \preceq_P [\sigma]$ has a unique maximal element $[\hat{\sigma}]$ and moreover $[\tau] \preceq [\hat \sigma]$ is canonical.
\end{definition}

\begin{definition}\label{def:co-p-sat}
Let $Y\subset \Simp(X)$ be any subset and let $P$ be any property on $X$. We call $Y$ {\em co-$P$-saturated} if $\tau\in Y$ and $\tau \precsim_P \sigma$ implies $\sigma \in Y$.  
\end{definition}

\begin{example}\label{ex:canon-sat}
We illustrate Definitions~\ref{def:the-canon} and~\ref{def:co-p-sat} for the symmetric $\Delta$-complex $X$ drawn in Figure~\ref{fig:master}. Here, the full set of symmetric orbits $\coprod_{p\ge 0}  X_p/S_{p+1}$ admits canonical co-$P$ maximal faces.  On the other hand, the $1$-skeleton $\coprod_{p=0,1} X_p/S_{p+1}$ does not: indeed, the poset $\{[\sigma]: [v]\precsim_P [\sigma]\}$ has two maximal elements.
 Finally, the set of simplices in the subcomplex generated by the $2$-simplex $T$ is co-$P$-saturated, while the vertex $v$ taken by itself is not co-$P$-saturated.
\end{example}

Given $P\subset X_0$ and an integer $i\ge 0$, let $X_{P,i}$ denote the subcomplex of $X$ generated by the set $V_{P,i}$ of $P$-simplices with at most $i$ non-$P$ vertices.
When no confusion seems possible, we write $X_{P,i}$ for the image of the natural map
\begin{equation}\label{eq:dpi}
\coprod_{p\ge 0}  \left ( X_p\cap V_{P,i} \right) \times \Delta^p  \longrightarrow |X|.
\end{equation}
For example, for $X$ and the property $Q$ as shown in Figure~\ref{fig:retract}, the subcomplexes $X_{Q,1}$ and $X_{Q,0}$ are shown in blue in Figure~\ref{fig:retract} (left and right sides respectively). 

We use $X_P$ to denote $X_{P,\infty}$, i.e., $X_P$ is the subcomplex generated by all $P$-simplices.  
In the specific case where $X=\Delta_{g,n}$, we abbreviate the notation and write
$$\Delta_{P,i} = (\Delta_{g,n})_{P,i}, \mbox{ \ \ and \ \ } \Delta_P = (\Delta_{g,n})_P.$$
Note that $\Delta_{P,i}$ parametrizes the closure of the locus of tropical curves with at least one $P$-edge and at most $i$ non-$P$ edges.  For instance, if $P$ is the property $P_{1,0}$ defined in Example~\ref{ex:properties-1}, then the subspace $\Delta_P \subset \Delta_{g,n}$ is the locus of tropical curves with either a loop or a vertex of positive weight.

We now state the main technical result of this section, which is a tool for producing deformation retractions inside symmetric $\Delta$-complexes.

\begin{figure}
\begin{tikzpicture}
\tikzstyle{vert}=[circle,thick,draw=gray!60,fill=gray!60,inner sep=1pt]
\tikzstyle{blue}=[circle,thick,draw=blue!60,fill=blue!60,inner sep=1pt]
\begin{scope}[scale=3]
\node (a)[vert] at (-.5,.86) {};
\node (b)[vert] at (0.5, .86) {};
\node (c)[vert] at (-1, 0) {};
\node (d)[vert] at (0, 0) {};
\node (e)[vert] at (1, 0) {};
\draw [gray!50, ultra thick, fill=gray!20] (a.center) -- (c.center) -- (d.center) -- (a.center);
\draw [gray!50, ultra thick, fill=gray!20] (a.center) -- (b.center) -- (d.center) -- (a.center);
\draw [gray!50, ultra thick, fill=gray!20] (e.center) -- (b.center) -- (d.center) -- (e.center);
\draw [blue!50, ultra thick, fill=blue!20] (e.center) -- (b.center) -- (d.center) -- (e.center);
\draw [blue!50, ultra thick] (a.center)--(d.center);
\draw [blue!50, ultra thick] (c.center)--(d.center);
\draw [blue!50, ultra thick, fill=blue!50] (d.center)--(e.center);
\node [blue] at (0, 0) {};
\node [blue] at (1, 0) {};
\draw [gray!60, ultra thick, <-] (-0.3,.65)--(-.3,.8);
\draw [gray!60, ultra thick, <-] (-0.15,.45)--(-.15,.8);
\draw [gray!60, ultra thick, <-] (0,.1)--(0,.8);
\draw [gray!60, ultra thick, <-] (0.15,.45)--(.15,.8);
\draw [gray!60, ultra thick, <-] (0.3,.65)--(.3,.8);
\draw [gray!60, ultra thick, <-, rotate around={60:(0,0)}] (-0.3,.65)--(-.3,.8);
\draw [gray!60, ultra thick, <-, rotate around={60:(0,0)}] (-0.15,.45)--(-.15,.8);
\draw [gray!60, ultra thick, <-, rotate around={60:(0,0)}] (0,.1)--(0,.8);
\draw [gray!60, ultra thick, <-, rotate around={60:(0,0)}] (0.15,.45)--(.15,.8);
\draw [gray!60, ultra thick, <-, rotate around={60:(0,0)}] (0.3,.65)--(.3,.8);
\node [below] at (0, 0) {$\scriptstyle{Q}$};
\node [below] at (1, 0) {$\scriptstyle{Q}$};
\end{scope}
\begin{scope}[shift={(9,0)},scale=3]
\node (a)[vert] at (-.5,.86) {};
\node (b)[vert] at (0.5, .86) {};
\node (c)[vert] at (-1, 0) {};
\node (d)[vert] at (0, 0) {};
\node (e)[vert] at (1, 0) {};
\draw [gray!50, ultra thick, fill=gray!20] (e.center) -- (b.center) -- (d.center) -- (e.center);
\draw [gray!50, ultra thick] (a.center)--(d.center);
\draw [gray!50, ultra thick, ->] (a.center)--(-.25,.43);
\draw [gray!50, ultra thick] (c.center)--(d.center);
\draw [gray!50, ultra thick, ->] (c.center)--(-.5,0);
\draw [gray!60, ultra thick, ->] (0.2,.25)--(.2,.05);
\draw [gray!60, ultra thick, ->] (0.35,.45)--(.35,.05);
\draw [gray!60, ultra thick, ->] (0.5,.75)--(.5,.05);
\draw [gray!60, ultra thick, ->] (0.65,.45)--(.65,.05);
\draw [gray!60, ultra thick, ->] (0.8,.25)--(.8,.05);
\node [below] at (0, 0) {$\scriptstyle{Q}$};
\node [below] at (1, 0) {$\scriptstyle{Q}$};
\draw [blue!50, ultra thick, fill=blue!50] (d.center)--(e.center);
\node [blue] at (0, 0) {};
\node [blue] at (1, 0) {};
\end{scope}

\end{tikzpicture}
\caption{The deformation retracts $X = X_{Q,2} \searrow X_{Q,1}$ and $X_{Q,1} \searrow X_{Q,0}$ as in Proposition~\ref{prop:prop}, with $P=\emptyset$. }
\label{fig:retract}
\end{figure}
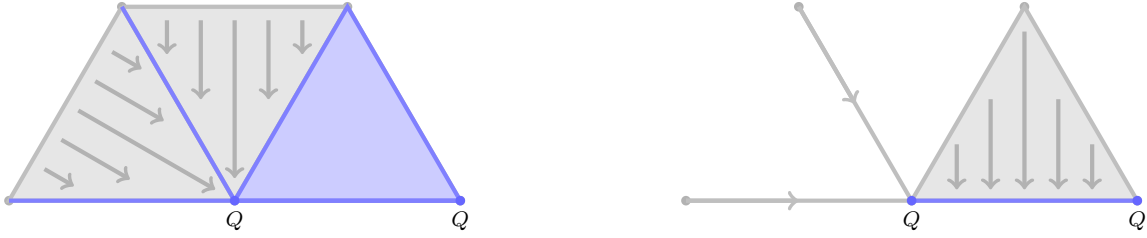

\begin{proposition}\label{prop:prop}
Let $X$ be a symmetric $\Delta$-complex.  Suppose $P,Q\subset X_0$ are properties satisfying the following conditions.

\begin{enumerate}
\item\label{it:(2)}  
The set of simplices $P^*(X)$ is co-$Q$-saturated.
\item\label{it:(1)}
The set of symmetric orbits of $X\smallsetminus P^*(X)$ admits canonical co-$Q$ maximal faces.
\end{enumerate}
Then there are strong deformation retracts
$(X_P\cup X_{Q,i})\searrow (X_P\cup X_{Q,i-1})$
for each $i>0$.  

If, in addition, every strictly $Q$-simplex is in $P^*(X)$, then there is a strong deformation retract $(X_P\cup X_Q) \searrow X_P$.
\end{proposition}

\noindent The basic ideas underlying Proposition~\ref{prop:prop} are discussed in Remark~\ref{rem:exegesis}, below.  We now give an example illustrating the conditions in  Proposition~\ref{prop:prop} on $X=\Delta_{g,n}$.  

\begin{example}\label{ex:properties-2}
Suppose $P\subset \Delta_{g,n}([0])$ is a property.  Because membership in $P(\Dgn)$ and $P^*(\Dgn)$ does not depend on edge-labeling, we say that $\Gmw$ is in $P$ if $(\Gmw,t)\in P(\Dgn)$ for any, or equivalently every, edge-labeling $t$.  Similarly, we say that $\Gmw$ is in $P^*$ if $(\Gmw,t)\in P^*(\Dgn)$ for any, or equivalently every, edge-labeling $t$.
 
Suppose $g>1$ and $n=0$, and let $P = P_{1,0}$ and $Q = P_{2,0}$, as defined in Example~\ref{ex:properties-1}.  
Note that $\Gmw \in P^*_{1,0}(X)$ if and only if $\Gmw$ has a loop or a positive vertex weight.
One may then check that $P$ and $Q$ satisfy the conditions of Proposition~\ref{prop:prop}.  The content is that every graph with a loop or weight, upon expansion by a $(2,0)$-bridge, still has a loop or weight; and that any graph with no loops or weights has a {\em canonical} expansion by $(2,0)$-bridges.  Furthermore, every strictly $Q$-simplex is in $P^*(\Dgn)$.   Then Proposition~\ref{prop:prop} asserts the existence of a deformation retraction $X_{P_{1,0}}\cup X_{P_{2,0}} \searrow X_{P_{1,0}}$.  In fact, this deformation retraction is the first step in the $n=0$ case of Theorem~\ref{thm:bridges}(\ref{item:1}).
\end{example}

\begin{example} \label{ex:nonexample-Dgn}
We give an example of a face inclusion in $\Dgn$ that is \emph{not} canonical.  Let $\Gmw$ and $\Gmw'$ be the graphs shown in Figure~\ref{f:canon}, on the left and right, respectively. 
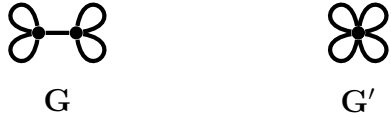
\begin{figure}[h!]
\begin{tikzpicture}[my_node/.style={fill, circle, inner sep=1.75pt}, scale=1]
\begin{scope}
\node[my_node] (A) at (-.25,0){};
\node[my_node] (B) at (.25,0){};
\draw[ultra thick] (A) to [out = 170, in = 100, looseness=15] (A);
\draw[ultra thick] (A) to [out = 190, in = 260, looseness=15] (A);
\draw[ultra thick] (A) to  (B);
\draw[ultra thick] (B) to [out = 10, in = 80, looseness=15] (B);
\draw[ultra thick] (B) to [out = -10, in = -80, looseness=15] (B);
\draw (0,-.9) node {$\Gmw$};
\end{scope}
\begin{scope}[shift = {(4,0)}]
\node[my_node] (A) at (0,0){};
\node[my_node] (B) at (0,0){};
\draw[ultra thick] (A) to [out = 170, in = 100, looseness=15] (A);
\draw[ultra thick] (A) to [out = 190, in = 260, looseness=15] (A);
\draw[ultra thick] (A) to  (B);
\draw[ultra thick] (B) to [out = 10, in = 80, looseness=15] (B);
\draw[ultra thick] (B) to [out = -10, in = -80, looseness=15] (B);
\draw (0,-.9) node {$\Gmw'$};
\end{scope} 
\end{tikzpicture}
\caption{The graphs $\Gmw$ and $\Gmw'$ on the left and right respectively. There is a morphism $\Gmw\rightarrow \Gmw'$ in $\J_{4,0}$ but it is {\em not} canonical with respect to $\Aut(\Gmw)$.}\label{f:canon}
\end{figure}
Note that $\Gmw'$ is isomorphic to a contraction of $\Gmw$.   Let us consider the equivalence relation on morphisms $\Gmw\rightarrow \Gmw'$ given by $\alpha_1 \sim \alpha_2$ if $\alpha_2 = \theta \alpha_1$ for some $\theta\in \Aut(\Gmw)$.  This equivalence relation partitions $\mathrm{Mor}(\Gmw,\Gmw')$ into exactly three classes, which are naturally in bijection with the three distinct unordered partitions of $E(\Gmw')$ into two groups of two.  In particular, there exist $\alpha_1,\alpha_2\colon \Gmw\rightarrow \Gmw'$ such that there is no $\theta\in \Aut(\Gmw)$ with $\alpha_2 = \theta \alpha_1$.
 Finally, by equipping $\Gmw$ and $\Gmw'$ appropriately with edge labelings $t$ and $t'$, respectively, this example can be promoted to an example of a face map in $\Dgn$ which is non-canonical.
This example shows that the full set of symmetric orbits of simplices in $\Dgn$ does {\em not} admit canonical co-$P_{2,0}$ maximal faces.  However, the set of symmetric orbits of $\Simp(\Dgn)\smallsetminus P_{1,0}^*(\Dgn)$ does, which is all that is required in Condition~\eqref{it:(1)}.
\end{example}

\begin{remark}\label{rem:exegesis}
We now sketch the idea of the proof of Proposition~\ref{prop:prop}, before proceeding to the proof itself.  Let us first assume that $P=\emptyset$.  In this case, Proposition~\ref{prop:prop} simplifies to the following: if $Q$ is a property such that the symmetric orbits of $X$ admit canonical co-$Q$ maximal faces, then there is a strong deformation retract $X_{Q,i} \searrow X_{Q,i-1}$ for each $i$.  These retractions are drawn in an example in Figure~\ref{fig:retract}.

Note that every $p$-simplex $\sigma \in V_{Q,i}$ which is not in $V_{Q,i-1}$ has precisely $p+1-i$ vertices in $Q$ and $i$ vertices not in $Q$.  
To such a simplex we shall associate a map $\Delta^p \to \partial \Delta^p$ by subtracting from the barycentric coordinates corresponding to vertices not in $Q$ and adding to the remaining ones, in a way that glues to a retraction $X_{Q,i} \to X_{Q,i-1}$.  Gluing the corresponding straight-line homotopies will give a homotopy from the identity to this retraction.  In order to carry this out, the main technical task is to verify that the different homotopies may in fact be glued, which is where Condition~\eqref{it:(1)} is used.

Now, dropping the condition that $P=\emptyset$ temporarily imposed in the previous paragraph, we obtain a {\em relative} version of the same argument.  In this case, Condition~\eqref{it:(2)} is needed in addition to guarantee that the relevant straight-line homotopies are constant on their overlap with $X_P$.  This relative formulation is useful to apply the proposition repeatedly over a sequence of properties $P_1,\ldots,P_n$. This sequential use of the proposition is packaged below as Corollary~\ref{cor:sequence}.
\end{remark}

The following definition and lemma will be used in the proof of Proposition~\ref{prop:prop}.  Let $P,Q\subset X_0$ be properties on $X$.
Recall that $Q(\sigma)$ and $Q^c(\sigma)$ denote the labels of the $Q$-vertices and non-$Q$-vertices of $\sigma$, respectively.
Let $i>0$, and let
\begin{equation}\label{eq:tqi-set}
S_i = \{\sigma\in V_{Q,i} :   \sigma\not\in P^*(X)\}.
\end{equation}
Let $T_i \subset S_i$ be the subset of simplices $\sigma$ whose symmetric orbits $[\sigma]$ are maximal with respect to the partial order $\preceq_Q$, and let $Y_i$ be the subcomplex of $X$ generated by $T_i$. 

\begin{definition}
Let $\sigma$ be a $p$-simplex of $X$ with $\sigma\in T_i$.  We define a homotopy $\rho_{\sigma,i}\colon \Delta^p\times[0,1]\to \Delta^p$ as follows. If $|Q^c(\sigma)|<i$ then $\rho_{\sigma,i}$ is the constant homotopy. Otherwise, 
define a map $$r_{\sigma,i}\colon\Delta^p\rightarrow \partial \Delta^p$$ as follows. Given $\ell \in \Delta^p$, let $$\gamma=\min_{e\in Q^c(\sigma)} \ell(e);$$ note that $\gamma=0$ is possible. Then define $r_{\sigma,i}(\ell)$ in coordinates by  
\begin{equation}\label{eq:cases}
r_{\sigma,i}(\ell)(e) = \begin{cases} \ell(e)-\gamma &\text{if } e\in Q^c(\sigma) \\ \ell(e)+\gamma i/|Q(\sigma)| &\text{if } e\in Q(\sigma).
\end{cases}
\end{equation}
(Note that $|Q(\sigma)|>0$ since $\sigma\in T_i$.)  Then let $\rho_{\sigma,i}$ be the straight line homotopy from the identity to $r_{\sigma,i}$.
\end{definition}

Now we prove two lemmas demonstrating that the homotopies $\rho_{\sigma,i}$ glue appropriately.  Write $\pi$ for the quotient map
\begin{equation}
 \pi\colon\Big(\coprod_{p = 0}^\infty X_p \times \Delta^p\Big) \rightarrow |X|
\end{equation}
 as in Equation~\eqref{eq:realization}, and write $\sim$ for the equivalence relation $(\sigma_1,\ell_1)\sim (\sigma_2,\ell_2)$ if $\pi(\sigma_1,\ell_1)=\pi(\sigma_2,\ell_2).$  We prove first that points in the inverse image of $X_P$ are fixed by each $\rho_{\sigma,i}$. 

\begin{lemma}\label{c:welldef1}
Let $X$ be a symmetric $\Delta$-complex and $P,Q\subset X_0$ properties on $X$ such that  $P^*(X)$ is co-$Q$-saturated.
Suppose $\sigma\in X_p$ is a simplex with $\sigma\in T_i$.  Given $\ell\in \Delta^p$, if $\pi(\sigma,\ell)\in X_P$ then $r_{\sigma,i}(\ell)=\ell$. Thus $\rho_{\sigma,i}(\ell, t)=\ell$ for all $t\in[0,1]$.
\end{lemma}

\begin{proof}[Proof of Lemma~\ref{c:welldef1}]
We prove the contrapositive, namely that if $r_{\sigma,i}(\ell)\ne \ell$ then $\pi(\sigma,\ell)$ is not in $X_P$.  In general, there exists some $q$, some $\tau\in X_q$ (uniquely determined up to $S_{q+1}$-action) and some $\ell'\in (\Delta^q)^\circ$ such that $(\sigma,\ell)\sim(\tau,\ell')$. Moreover, the assumption $r_{\sigma,i}(\ell)\ne \ell$ implies that $\gamma>0$ in Equation~\eqref{eq:cases}.    Therefore, $\tau\precsim_Q \sigma$.  Now $\sigma\not\in P^*(X)$ since $\sigma\in T_i$, so $\tau \not\in P^*(X)$ since $P^*(X)$ is co-$Q$ saturated.  Therefore $\pi(\sigma,\ell)=\pi(\tau,\ell') \not \in X_P$. 
\end{proof}

Next, we prove that the homotopies $\rho_{\sigma,i}$ agree on overlaps, as $\sigma$ ranges over $T_i$.

\begin{lemma}\label{c:welldef2}
Let $X$ be a symmetric $\Delta$-complex and $P,Q\subset X_0$ properties such that  
\begin{enumerate}
\item $P^*(X)$ is co-$Q$-saturated, and
\item the set of symmetric orbits of $X\smallsetminus P^*(X)$ admits canonical co-$Q$ maximal faces.
\end{enumerate}
Given $\sigma_1\in X_{p_1}$ and $\sigma_2\in X_{p_2}$ with $\sigma_1,\sigma_2\in T_i$, suppose $\ell_1\in \Delta^{p_1}$ and $\ell_2\in \Delta^{p_2}$ are such that $(\sigma_1,\ell_1)\sim(\sigma_2,\ell_2)$. Writing $r_1 = r_{\sigma_1,i}$ and $r_2=r_{\sigma_2,i}$ for short, we have $$(\sigma_1, r_1(\ell_1)) \sim(\sigma_2, r_2(\ell_2)).$$ 
Therefore $(\sigma_1, \rho_{\sigma_1,i}(\ell_1,t)) \sim(\sigma_2, \rho_{\sigma_2,i}(\ell_2,t))$  for all $t\in[0,1]$. 
\end{lemma}
 
\begin{proof}[Proof of Lemma~\ref{c:welldef2}]
Again, there exists some $q$, some $\tau\in X_q$ (uniquely defined up to $S_{q+1}$-action) and some $\ell\in(\Delta^q)^\circ$ such $(\tau,\ell)\sim(\sigma_1,\ell_1)\sim(\sigma_2,\ell_2)$. Moreover if $\tau\in P^*(X)$ then we are done by Claim~\ref{c:welldef1}, so we assume $\tau\not\in P^*(X)$.  There are two cases.

First, if $|Q^c(\tau)|<i$, then for each $j=1,2$, either $\tau \not \precsim_Q \sigma_j$ or $\tau \precsim_Q\sigma_j$; in the first case we have $\min_{e \in Q^c(\sigma_j)} \ell(e) =0$, and in the second case we have $|Q^c(\sigma_j)| = |Q^c(\tau)| < i.$ Thus in both cases, 
 $r_1(\ell_1) = \ell_1$ and $r_2(\ell_2) = \ell_2,$ which proves the claim.

Second, if $|Q^c(\tau)|=i$, we have $[\tau]\preceq_Q[\sigma_1]$ and $[\tau]\preceq_Q[\sigma_2]$. In fact, for $j=1,2$, we claim $[\sigma_j]$ is maximal such that $[\tau]\preceq_Q[\sigma_j]$.  Indeed, if $[\tau]\preceq_Q [\sigma_j]\prec_Q [\sigma]$ for some $[\sigma]$, then
\begin{itemize}
\item $|Q(\sigma)|>0$, since $\sigma_j \precnsim_Q \sigma$;
\item $|Q^c(\sigma)|\le i$, since $|Q^c(\sigma)|=|Q^c(\sigma_j)|\le i$;
\item $\sigma\not \in P^*(X)$, since $\sigma_j\not \in P^*(X)$.
\end{itemize}
But this contradicts that $\sigma_j\in T_i$.  We note again that $\tau\not\in P^*(X)$ by assumption.  
Therefore $[\sigma_1]=[\sigma_2]$, by the hypothesis that the symmetric orbits of $X\smallsetminus P^*(X)$ admit canonical co-$Q$ maximal faces. 

Let us treat the special case $\sigma_1=\sigma_2$; the general case will follow easily from it.  Write $\sigma = \sigma_1=\sigma_2$ and $r = r_1=r_2$.  
For $j=1,2$, since $(\tau,\ell)\sim(\sigma,\ell_j)$, there exists $\alpha_j\colon [q]\hookrightarrow[p] $ such that $\alpha_j^*\sigma = \tau$ and $\ell_j = \alpha_{j*}\ell$.  By canonicity of co-$Q$-maximal faces, there exists $\theta\in\Aut(\sigma)$ with $\theta\alpha_1=\alpha_2$, so 
$$\ell_2 = \alpha_{2*}\ell = \theta_*\alpha_{1*}\ell = \theta_*\ell_1 = \ell_1\circ \theta.$$  Since $\theta\in\Aut(\sigma)$, $\theta e \in Q(\sigma)$ if and only if $e\in Q(\sigma)$ for all $e\in[p]$.  Therefore by Equation~\eqref{eq:cases} we have
$$(\sigma,r(\ell_1)) \sim (\sigma,r(\ell_1\circ \theta)) \sim (\sigma, r(\ell_2)),$$
as desired.

Finally, the general case follows from the previous one by replacing $\sigma_2$ with $\sigma_1 = \phi^*\sigma_2$ for some $\phi\colon[p]\rightarrow [p]$, and replacing $\ell_2$ with $\ell_2\circ\phi$.  Indeed, we have $(\sigma_1,\ell_1)\sim(\sigma_2,\ell_2)\sim(\sigma_1,\ell_2\circ\phi)$ and 
\begin{eqnarray*}
(\sigma_2,r_2(\ell_2))) &\sim& (\phi^*\sigma_2, r_2(\ell_2)\circ \phi)\\
&\sim&(\sigma_1, r_1(\ell_2\circ\phi))\\ 
&\sim&(\sigma_1,r_1(\ell_1)).
\end{eqnarray*}
The second equivalence follows from the fact that $e\in Q(\sigma_2)$ if and only if $\phi(e) \in Q(\phi^*\sigma_2) = Q(\sigma_1)$ for every $e\in[p]$. The last equivalence follows from the previous computation for the case $\sigma_1=\sigma_2$.   
This proves Claim~\ref{c:welldef2}.
\end{proof}

We now proceed with the proof of Proposition~\ref{prop:prop}.

\begin{proof}[Proof of Proposition~\ref{prop:prop}]
Let $X$ be a symmetric $\Delta$-complex, and let $P,Q\subset X_0$ be properties satisfying:

\begin{enumerate}
\item $P^*(X)$ is co-$Q$-saturated, and
\item the symmetric orbits of $X\smallsetminus P^*(X)$ admit canonical co-$Q$ maximal faces.
\end{enumerate}
We wish to exhibit a deformation retract $X_P\cup X_{Q,i} \searrow X_P\cup X_{Q,i-1}.$

Recall that $Y_i$ is the subcomplex of X generated by the set of simplices $T_i$; by the usual abuse of notation we will also write $Y_i$ for the homeomorphic image of its geometric realization in $|X|$.
First we note $X_P \cup X_{Q,i} = X_P \cup Y_i.$  The inclusion $\supset$ is clear since $\Simp(X_{Q,i})\supset T_i$. The inclusion $\subset$ is also apparent: suppose $\tau\in \Simp(X)$ has $|Q(\tau)|>0$ and $|Q^c(\tau)|\le i$.  If $\tau\in P^*(X)$ then its image in $|X|$ is in $X_P$.  Otherwise, $\tau\in S_i$, so $\tau \precsim_Q \sigma$ for some $\sigma\in T_i$, so the image of $\tau$ in $|X|$ lies in $Y_i$.

Therefore, we have a map $$r_{i}\colon (X_P\cup X_{Q,i}) \rightarrow X_P\cup X_{Q,i}$$
that is obtained by gluing the maps $r_{\sigma,i}$ for $\sigma\in T_i$, together with the constant map on $X_P$.  The fact that we may glue these maps together is the content of Lemmas~\ref{c:welldef1} and~\ref{c:welldef2}.   Moreover $r_i$ restricts to the constant map on $X_P\cup X_{Q,i-1}$ by construction.  

Next, we show that the image of $r_i$ is $X_P\cup X_{Q,i-1}$.  Let  $\sigma\in T_i$ be a $p$-simplex and let $\ell\in \Delta^p$.   Now 
there exists some $q$, some $\tau\in X_q$, and some $\ell'\in (\Delta^q)^\circ,$ such that $(\sigma,r_i(\ell))\sim(\tau,\ell')$.   
Examining~\eqref{eq:cases} shows that $|Q(\tau)|>0$ and $|Q^c(\tau)|<|Q^c(\sigma)|=i$.  Now if $\tau \in P^*(X)$, then $\pi(\sigma,r_i(\ell))\in X_P$. Otherwise, if $\tau\not\in P^*(X)$, then $\tau\in S_{i-1}$, so $\pi(\sigma,r_i(\ell))\in X_{Q,i-1}$.  This argument shows that the map $\rho_{i}\colon (X_P\cup X_{Q,i}) \times [0,1]\rightarrow X_P\cup X_{Q,i}$, defined to be the straight line homotopy associated to $r_i$, is a deformation retract onto $X_P\cup X_{Q,i-1}$.  
Thus we have a strong deformation retract $X_P\cup X_{Q,i} \searrow X_P\cup X_{Q,i-1}$ for each $i$, and hence a strong deformation retract $X_P\cup X_{Q} \searrow X_P\cup X_{Q,0}$.

Finally, we check that if every strictly $Q$-simplex is in $P^*(X)$, then $X_{Q,0} \subset X_P$.  Indeed, if this condition holds, then $$\Simp(X_{Q,0}) = \{\sigma\in \Simp(X) \ | \ Q^c(\sigma)=\emptyset\} \subset \Simp(X_P) = P^*(X),$$ so $X_{Q,0}\subset X_P$ as desired.  Thus, under this condition there is a strong deformation retract $X_P\cup X_Q \searrow X_P$, finishing the proof of the proposition.
\end{proof}

We record an obvious corollary of Proposition~\ref{prop:prop}, obtained by applying it repeatedly.

\begin{corollary}\label{cor:sequence}
Let $X$ be a symmetric $\Delta$-complex, and let $P_1,\ldots,P_N$ be a sequence of properties.
\begin{enumerate}
\item Suppose that for $i=2,\ldots,N$, the two properties $P=P_1\cup\cdots\cup P_{i-1}$ and $Q=P_i$ satisfy that
\begin{itemize}
\item $P^*(X)$ is co-$Q$-saturated, 
\item the symmetric orbits of $X\smallsetminus P^*(X)$ admit canonical co-$Q$ maximal faces, and
\item every strictly $Q$-simplex is in $P^*(X)$.
\end{itemize}
Then there exists a strong deformation retract
$$X_{P_1\cup\cdots\cup P_N}\searrow X_{P_1}.$$
\item If in addition 
the symmetric orbits of $X$ admit canonical co-$P_1$ maximal faces, then
there exists a strong deformation retract
$$X_{P_1\cup\cdots\cup P_N}\searrow X_{P_1,0}.$$ 
\end{enumerate}
Here the spaces $X_P$ and $X_{P,0}$, for a property $P$, are the ones defined in~\eqref{eq:dpi}.
\end{corollary}

\subsection{Contractible subcomplexes of $\Dgn$}\label{subsec:blocks-etc}

Here, we prove contractibility of three natural subcomplexes of $\Dgn$.  First, recall that a {\em bridge} of a connected graph $G$ is an edge whose deletion disconnects $G$, and let $\br \subset \Dgn$ denote the closure of the locus of tropical curves with bridges.  It is the geometric realization of the subcomplex generated by those $\Gmw\in \J_{g,n}$ with bridges.  Next, we say that $\Gmw=(G,m,w)\in \J_{g,n}$ has {\em repeated markings} if the marking function $m$ is not injective.  Let $\rep \subset \Dgn$ be the locus of tropical curves with repeated markings.   Let $\weight$ be the locus of tropical curves with at least one vertex of positive weight, and let $\lw$ be the locus of tropical curves with loops or vertices of positive weights.   

Note that $\weight$ is a closed subcomplex of $\lw$ and both $\rep$ and $\lw$ are closed subcomplexes of $\br$.  For some purposes, it is most useful to contract the largest possible subcomplex.  However, contractibility of smaller subcomplexes is also valuable; for instance, the contractibility of $\weight$ allows us to identify the reduced homology of $\Dgn$ with graph homology in Theorem~\ref{thm:split}.

We recall the statement of Theorem~\ref{thm:contractible} assuming, as throughout, that $2g - 2 + n > 0$.

\contractible*

We will prove this theorem by applying Corollary~\ref{cor:sequence} to a particular sequence of properties, as follows.  Let $\Gmw=(G,m,w)\in \J_{g,n}$.  Recall from Example~\ref{ex:properties-1} that an edge $e \in E(G)$ is a {\em $(g',n')$-bridge} if $\Gmw/(E(G)-e) \cong\bfB(g',n')$, i.e., a $(g',n')$-bridge separates $\Gmw$ into subgraphs of types $(g',n')$ and $(g-g',n-n')$, respectively. 
We write $P_{g',n'}$ for the property $\{\bfB(g',n')\}.$

\begin{theorem}\label{thm:bridges}
Let $g>0$ and $X = \Delta_{g,n}$.
\begin{enumerate} 
\item If $n\ge 2$, then the sequence of properties
$$P_{0,n}, P_{0,n-1}, \ldots, P_{0,2}, P_{1,0}, P_{1,1},\ldots, P_{1,n}, P_{2,0}, \ldots P_{2,n},\ldots$$
satisfies both conditions of Corollary~\ref{cor:sequence}.  
\item\label{item:1} If $n=0$ or $1$ and $g>1$, then the sequence of properties
$$P_{1,0}, \ldots, P_{1,n}, P_{2,0}, \ldots P_{2,n},\ldots$$
satisfies both conditions of Corollary~\ref{cor:sequence}.  
\end{enumerate}
Each of these sequences of properties is finite.  The last term of each of the two sequences above is chosen so that each type of bridge is named once.  Precisely, if $g$ is even, the last term is $P_{g/2, \lfloor n/2\rfloor}$; if $g$ is odd, the last term is $P_{({g-1})/{2}, n}$.  
\end{theorem}
  
\begin{remark}\label{rem:when-nonempty}
With the standing assumption that $g>0$, the loci $\Dgn^{\rm lw}$ are never empty, and $\weight$ is empty only when $(g,n) = (1,1)$.  The locus $\Dgn^{\rm rep}$ is empty exactly when $n \leq 1$.  The locus $\Dgn^{\rm br}$ is empty exactly when $(g,n) = (1,1)$.  Otherwise, it contains $\Dgn^{\rm lw} \cup \Dgn^{\rm rep}$.
\end{remark}

\begin{proof}[Proof that Theorem~\ref{thm:bridges} implies Theorem~\ref{thm:contractible}]
First we show Theorem~\ref{thm:contractible}\eqref{it:br-contract}. We treat two cases: if $n\ge 2$, let $P_1,\ldots,P_N$ denote the sequence of properties in part (1) of Theorem~\ref{thm:bridges}(1); if $n\le 1$ and $g>1$, let $P_1,\ldots,P_N$ denote the sequence of properties in part (2) of Theorem~\ref{thm:bridges}.  In either case, $\cup P_i$ is the property of being a bridge, so $\Delta_{\cup P_i} = \Delta^\mathrm{br}_{g,n}$.  In the first case, $P_1=P_{0,n}$ is the property of being a $(0,n)$-bridge, and note that $\Delta_{P_{0,n},0}$ is a point: there is a unique (up to isomorphism) tropical curve whose edges are all $(0,n)$-bridges.  In the second case, $P_1=P_{1,0}$ is the property of being a $(1,0)$-bridge, and $\Delta_{P_{1,0},0}$ is a $(g\!-\!1)$-simplex, parametrizing nonnegative edge lengths on a tree with $g$ leaves of weight 1, and a central vertex supporting $n$ markings.  
Then by Theorem~\ref{thm:bridges}, we may apply Corollary~\ref{cor:sequence} to produce a deformation retract from $\Delta_{\cup P_i} = \Delta^\mathrm{br}_{g,n}$ to a contractible space.  
This shows Theorem~\ref{thm:contractible}\eqref{it:br-contract}.

We deduce Theorem~\ref{thm:contractible}\eqref{it:rep-contract} by considering only the subsequence of properties $P_1 = P_{0,n}, \ldots, P_{n-1}=P_{0,2}$. Indeed, $P_{0,n},\ldots,P_{0,2}$, being an initial subsequence of the properties listed in Theorem~\ref{thm:bridges}(1), also satisfies both conditions of Corollary~\ref{cor:sequence}.  Moreover $\Delta_{\cup P_i} = \rep$ and $\Delta_{P_{0,n},0}$ is a point.  So by Corollary~\ref{cor:sequence}, we conclude that $\rep$ is contractible for all $g>0$ and $n>1$.

For Theorem~\ref{thm:contractible}\eqref{it:loop-contract}, if $(g,n)=(1,1)$ the claim is trivial.  
Else, we verify directly that the properties $P=\emptyset$ and $Q=P_{1,0}$ satisfy the conditions~\eqref{it:(2)} and~\eqref{it:(1)} of Proposition~\ref{prop:prop}.  In other words, we verify directly that every graph admits a canonical maximal expansion by $(1,0)$-bridges.  If $\Gmw$ has no loops or weights, the expansion is trivial. Otherwise, the expansion is as follows: for any vertex $v$ with $$\val(v) + 2w(v) >3,$$  
replace every loop based at $v$ with a bridge from $v$ to a loop; add $w(v)$ bridges to vertices of weight $1$, and set $w(v)=0$.
Contractibility of the loop-and-weight locus follows from Proposition~\ref{prop:prop}, noting that this locus is exactly the subcomplex $(\Delta_{g,n})_{1,0}$ of $\Dgn$ whenever $(g,n)\ne (1,1)$.

The proof of Theorem~\ref{thm:contractible}\eqref{it:weight-contract} is similar.  If $(g,n)=(1,1)$ then $\weight$ is empty.  Otherwise, $\Delta^\mathrm{w}_{g,n}$ itself as a symmetric $\Delta$-complex, and we may consider the properties $P=\emptyset$ and $Q=P_{1,0}$ on it. This pair of properties still satisfies the conditions~\eqref{it:(2)} and~\eqref{it:(1)} of Proposition~\ref{prop:prop}, since any $\Gmw$ with a vertex of positive weight has a canonical maximal expansion of $\Gmw$ by $(1,0)$-bridges that again has a vertex of positive weight whenever $\Gmw$ does; this expansion is described above.  So by Proposition~\ref{prop:prop}, $\weight = (\weight)_{P_{1,0}}$ deformation retracts down to the subcomplex of $\weight$ consisting of graphs in which the only edges are $(1,0)$-edges, and this subcomplex is contractible.
\end{proof}

In order to prove Theorem~\ref{thm:bridges}, it will be convenient to develop a theory of {\em block decompositions} of stable weighted, marked graphs.  Let us start with usual graphs, without weights or markings.  If $G$ is a connected graph, we say $v\in V(G)$ is a cut vertex if deleting it disconnects $G$.  A {\em block} of $G$ is a maximal connected subgraph with at least one edge and no cut vertices. 

\begin{example}\label{ex:blocks} If $G$ is a graph on two vertices $v_1,v_2$ with a loop at each of $v_1$ and $v_2$ and $n$ edges between $v_1$ and $v_2$, then $G$ has three blocks: the loop at $v_1$, the loop at $v_2$, and the $n$ edges between $v_1$ and $v_2$.
\end{example}  

Returning to marked, weighted graphs, we define an {\em articulation point} of a stable, marked weighted graph $\Gmw = (G,m,w)$ to be a vertex $v\in V(G)$ such that at least one of the following conditions holds:
\begin{enumerate}[(i)]
\item $v$ is a cut vertex of $G$,
\item $w(v)>0$, or
\item $|m^{-1}(v)|\ge 2$.
\end{enumerate}
These are an analogue of cut vertices for marked, weighted graphs.  Let $\mathcal{A}$ denote the set of articulation points, and let $\mathcal{B}$ denote the set of blocks of the underlying graph~$G$.

\begin{definition}\label{def:blockgraph}
Let $\Gmw$ be a weighted, marked graph.  The {\em block graph} of $\Gmw$, denoted $\Bl(\Gmw)$, is a graph defined as follows. The vertices are $\mathcal{A}\cup \mathcal{B}$, and there is an edge $E=(v,B)$ from $v\in \mathcal{A}$ to $B\in\mathcal{B}$ if and only if $v \in B$.
\end{definition} 

\noindent In this way $\Bl(\Gmw)$ is naturally a tree, whose vertices are articulation points and blocks.  The block graph of the graph $G$ in Example~\ref{ex:blocks} is drawn in Figure~\ref{fig:blocks}. The vertices of the block graph are depicted as the blocks and articulation points to which they correspond. The edges of the block graph are drawn in blue.

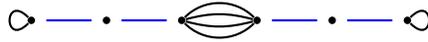
\begin{figure}[h!]
\begin{tikzpicture}[my_node/.style={fill, circle, inner sep=1pt}, scale=1]
\node[my_node] (L) at (-4,0){};
\draw[thick] (L) to [out = 135, in = 225, looseness=15] (L);
\node[my_node] (A1) at (-3,0){};
\node[my_node] (V1) at (-2,0){};
\node[my_node] (V2) at (-1,0){};
\draw[thick] (V1) to [bend right=45] (V2);
\draw[thick] (V1) to [bend right=15] (V2);
\draw[thick] (V1) to [bend left=15] (V2);
\draw[thick] (V1) to [bend left=45] (V2);
\node[my_node] (A2) at (0,0){};
\node[my_node] (R) at (1,0){};
\draw[thick] (R) to [out = 45, in = 315, looseness=15] (R);
\draw[blue, thick] (-3.8, 0) -- (-3.2,0);
\draw[blue, thick] (-2.8, 0) -- (-2.2,0);
\draw[blue, thick] (-.8, 0) -- (-.2,0);
\draw[blue, thick] (.2, 0) -- (.8,0);
\end{tikzpicture}
\caption{Block graph of $G$ as in Example~\ref{ex:blocks}.}\label{fig:blocks}
\end{figure}

At this point, we will equip both the articulation points and the blocks with weights and markings on the vertices, according to the following conventions.  If $v\in \mathcal{A} $ is an articulation point, we take it to have the weight and markings it has in $\Gmw$.  That is, $v$ has weight $w(v)$ and markings $m^{-1}(v)$.  If $B\in\mathcal{B}$ is a block, then we give each vertex $x\in V(B)$ weights and markings according to the following rule.  If $x\in\mathcal{A}$ then we equip it with weight $0$ and no markings.  Otherwise, we equip $x$  with the same weights and markings as it had in $\Gmw$.  In this way, we now regard each articulation point and each block as a weighted marked graph. We emphasize that these weighted marked graphs need not be stable.
We note
 $$\sum_{\mathbf{H} \in V(\Bl(\Gmw))} g(\mathbf{H}) = g,\qquad \sum_{\mathbf{H} \in V(\Bl(\Gmw))} n(\mathbf{H}) = n.$$
(Here $n(\mathbf{H})$ is the number of marked points.)  

It will be useful to label the edges of $\Bl(\Gmw)$ as follows.  Since $\Bl(\Gmw)$ is a tree, deleting any edge $\epsilon=(\bfv, B)$ divides $\Bl(\Gmw)$ into two connected components. Let $S$ be the set of vertices in the part containing $B\in V(\Bl(\Gmw))$; then we label the edge $\epsilon$ 
$$(g(\bfv,B),n(\bfv,B)):=\left(\sum_{\mathbf{H}\in S} g(\mathbf{H}), \, \sum_{\mathbf{H}\in S} n(\mathbf{H})\right).$$
A property of this labeling that we record for later use is that for every $\bfv\in \mathcal{A}$, 
\begin{equation}\label{eq:sum-around-art}
\sum_{B \ni v} g((\bfv,B)) + w(v) = g, \mbox{ \ \ and \ \ } \sum_{B \ni v} n((\bfv,B)) + |m^{-1}(v)| = n. 
\end{equation}

\begin{example}
Let $g>0$ and $n\ge 0$ with $(g,n)\ne (1,0),(1,1)$.  Suppose $\Gmw=(G,m,w)$ has a single vertex $v$ and $h$ loops.  Then $v$ has weight $g-h$ and $n$ markings, and there are $h$ blocks of $\Gmw,$ each a single unweighted, unmarked loop based at $v$. There is a single articulation point $\bfv$, equipped with weight $g-h$ and all $n$ markings.  The block graph $\Bl(\Gmw)$ is a star tree with $h$ edges from $\bfv$, each labeled $(1,0)$.
\end{example}

We make the following observations.

\begin{lemma}\label{lem:art} Let $\Gmw\in \J_{g,n}$.
\begin{enumerate}
\item \label{it:contract-to-art}
If $e\in E(G)$ is a bridge then its image vertex $v$ in $\Gmw/e$ is an articulation point.
\item \label{it:expand-art} Let $\bfv$ be an articulation point of $\Gmw$, with weight $u\ge 0$ and markings $m^{-1}(v) = M$, and with edges of $\Bl(\Gmw)$ at $\bfv$ labeled $(g_1,n_1),\ldots,(g_s,n_s)$.  Then $\bfv$ may be expanded into a bridge, with the result a stable marked, weighted graph, in any of the following ways.  Choose a partition of the edges of $\Bl(\Gmw)$ at $\bfv$ into two parts $P_1$ and $P_2$; choose a partition of the set $M$ into sets $M_1$ and $M_2$; and choose integers $w_1,w_2\ge 0$ with $w_1+w_2=u$, such that for $j=1,2$
$$\sum_{(\bfv,B)\in P_j} \val_B(v) + |M_j| + 2w_j \ge 2.$$
Here and below, $\val_B(v)$ denotes the number of half-edges at $v$ lying in $B$; it does not count any marked points.
By dividing the blocks, markings, and weight accordingly, $\bfv$ may be expanded into a bridge of type
\begin{equation}\label{eq:bridge-type}
  \left(\sum_{(\bfv, B)\in P_1} g((\bfv,B)) + w_1, \sum_{(\bfv, B)\in P_1} n((\bfv,B)) + |M_1|\right)
  \end{equation}
such that the result is stable; and no other stable expansions of $\bfv$ into bridges are possible.
\item \label{it:actually-spec}
If $\Bl(\Gmw)$ has an edge $\epsilon=(\bfv,B)$ labeled $(g',n')$, then $\Gmw \in P_{g',n'}^*$.
\item\label{it:converse}
Suppose $g'\ge 1$ and $w(v)=0$ for all $v\in V(G)$, and suppose every label $(g'',n'')$ on $E(\Bl(\Gmw))$ satisfies either $g''>g'$, or $g''=g'$ and $n''>n'$.  Then $\Gmw \not \in P_{g',n'}^*$.
\end{enumerate}
\end{lemma}

\begin{proof}
Statements~\eqref{it:contract-to-art} and~\eqref{it:expand-art} are easy to check.  Statement~\eqref{it:converse} then follows: if $\Gmw\in P_{g',n'}^*$ then (1) and (2) imply that some articulation point $v$ may be expanded into a bridge of type $(g',n')$, with $$(g',n') =  \left(\sum_{(\bfv, B)\in P_1} g((\bfv,B)) + w_1, \sum_{(\bfv, B)\in P_1} n((\bfv,B)) + |M_1|\right)$$ 
for some choice of partition $P_1\sqcup P_2$ of the blocks at $v$.  Since $g' > 0$ and $w_1 = w(v) = 0$ we must have $P_1\ne \emptyset$, but then the expression in~\eqref{eq:bridge-type} exceeds $(g',n')$ in lexicographic order.

  For statement~\eqref{it:actually-spec}, suppose $\epsilon=(\bfv, B)$ is labeled $(g',n')$. If $B$ itself is a $(g',n')$-bridge we are done. Otherwise $\val_{B} (v) \ge 2$.  Write $B_1,\ldots,B_s$ for the remaining blocks at $\bfv$.  If $\sum_{j=1}^s \val_{B_j} (v) + |m^{-1}(v)| + 2w(v) \ge 2$ then $\bfv$ can be expanded into a $(g',n')$-bridge by~\eqref{it:expand-art}.  So assume  $$\sum_{j=1}^s \val_{B_j} (v) + |m^{-1}(v)| + 2w(v) \le 1.$$  The only possibility consistent with $\bfv$ being an articulation point is $s=1$, $\val_{B_1}(v) = 1$, $|m^{-1}(v)|=0$, and $w(v) = 0$.  Thus $B_1$ is a bridge, and the identities~\eqref{eq:sum-around-art} show that $B_1$ is a $(g\!-\!g',n\!-\! n')$-bridge, which is the same as a $(g',n')$-bridge.
\end{proof}

Now we turn to the proof of Theorem~\ref{thm:bridges}.

\begin{proof}[Proof of Theorem~\ref{thm:bridges}]
Fix $g>0$ and $n\ge 0$.  If $n\ge 2$, let $P_1,P_2,\ldots$ be the sequence of properties
$$P_{0,n},\ldots,P_{0,2},P_{1,0}, \ldots,P_{1,n},P_{2,0},\ldots,P_{2,n},\ldots.$$
If $n=0$ or $1$ and $g\ne 1$, let $P_1,P_2,\ldots$ be the sequence of properties
$$P_{1,0}, \ldots,P_{1,n},P_{2,0},\ldots,P_{2,n},\ldots.$$
We need to check:
\begin{enumerate}[(i)]
\item for each $i=2,3,\ldots$ the properties $P=P_1\cup \cdots \cup P_{i-1}$ and $Q=P_i$ satisfy the two conditions of Proposition~\ref{prop:prop}, and every strictly $Q$-simplex is in $P^*$.
\item the symmetric orbits of $X$ admit canonical co-$Q$-maximal faces. 
\end{enumerate}
\noindent Item (ii) above is exactly the statement that the properties $P=\emptyset$ and $Q=P_1$ satisfy the second condition of Proposition~\ref{prop:prop}. 

{\bf Condition~\eqref{it:(1)} of Proposition~\ref{prop:prop}.}  For each $i=1,2,\ldots,$ let $P=P_1\cup\cdots\cup P_{i-1}$ and $Q=P_i$. Let us check that condition~\eqref{it:(1)} of Proposition~\ref{prop:prop} holds.    Let $Q = P_{g',n'}$. Suppose $\Gmw\in \J_{g,n}$ is not in $P^*$. We need to show that $\Gmw$ admits a maximal uncontraction $\alpha\colon \widetilde{\Gmw}\rightarrow\Gmw$ by $(g',n')$-bridges, which is canonical in the sense that for any $\alpha'\colon \widetilde{\Gmw}\rightarrow\Gmw$, there exists an automorphism $\theta\colon \widetilde{\Gmw}\rightarrow \widetilde{\Gmw}$ such that $\alpha'\theta=\alpha$.  Informally speaking, we are saying that $\widetilde{\Gmw}$ may be described in a way that is intrinsic to $\Gmw$.  We treat three cases: 
\begin{itemize}
\item $Q=P_{g',n'}$ with $g'\ge 1$ and $(g',n')\ne(1,0)$; 
\item $Q=P_{1,0}$; and   
\item $Q=P_{0,n'}$ for some $n'$.  
\end{itemize}
The case $Q=P_{0,n'}$ is only needed when $n\ge 2$.

First, assume $Q=P_{g',n'}$ with $g'\ge 1$ and $(g',n')\ne(1,0)$.  Let $\bfv$ be any articulation point.  
Now either $n\le 1$, or $n\ge 2$ and $\Gmw$ is assumed not to be a $(0,2)$-contraction since $P_{0,2}\subset P$. Therefore $\Gmw$ has no repeated markings. Since $P_{1,0}\subset P$ and $\Gmw$ is assumed not to be in $P^*(X)$, $\Gmw$ has no vertex weights.
 Let $B\in\Bl(\bfv)$, and let $(g'',n'')$ be the label of $\epsilon=(v,B)\in E(\Bl(\Gmw))$.  Then by Lemma~\ref{lem:art}\eqref{it:actually-spec}, $\Gmw \in P_{g'',n''}^*$. Referring to the chosen ordering of properties, it follows that $g''\ge g'$, and if $g''=g'$ then $n''>n'$.  Now using the criterion on Lemma~\ref{lem:art}\eqref{it:expand-art}, we conclude that the only $(g',n')$-bridge expansions admitted at $\bfv$ are along the pairs $(v,B)$ labeled exactly $(g',n')$ where $B$ is not itself a bridge, and there is a unique maximal such expansion which is canonical in the previously described sense.

Second, assume $Q=P_{1,0}$.  Then by Lemma~\ref{lem:art}\eqref{it:expand-art}, the maximal $(1,0)$-bridge expansion of $\Gmw$ is obtained by replacing, for any vertex $v$ with $\val(v) + 2w(v) >3$, every loop based at $v$ with a bridge from $v$ to a loop; adding $w(v)$ bridges to vertices of weight 1, and setting $w(v)=0$.  Moreover this expansion is canonical.

Finally, assume $Q=P_{0,n'}$; this case is only needed when $n\ge 2$.  Consider an articulation point $\bfv$, and let $B_1,\ldots,B_k$ be the blocks at $v$ labelled $(0,n_1),\ldots,(0,n_k)$ for some $n_i$. We are assuming that $\Gmw$ is not in $P^*$; in this case the chosen ordering of properties implies that $P_{0,n''}\subset P$ for each $n''>n'.$  Therefore, by Lemma~\ref{lem:art}\eqref{it:expand-art}, $\sum n_i + |m^{-1}(v)| \le n'$. Furthermore, $\bfv$ can be expanded into a $(0,n')$-bridge if and only if equality holds, so long as it is not the case that $k=1$ and $B_1$ is itself a $(0,n')$-bridge.  This analysis, performed at all articulation points, produces the unique maximal $(0,n')$-bridge expansion of $\Gmw$, and this expansion is canonical.  This verifies condition~\eqref{it:(1)} of Proposition~\ref{prop:prop}.

{\bf Condition~\eqref{it:(2)}   of Proposition~\ref{prop:prop}}.  Again, let $i=1,2,\ldots,$ let $P=P_1\cup\cdots\cup P_{i-1}$ and $Q=P_i$; we now check that condition~\eqref{it:(2)} of Proposition~\ref{prop:prop} holds.  Suppose $Q = P_{g',n'}$.   We want to show that if $\Gmw\in \J_{g,n}$ is not in $P^*$ and $\Gmw'$ is obtained by contracting $(g',n')$-bridges, then $\Gmw'$ is also not in $P^*$.  We consider the same three cases.

First, assume $g'=0$, that is, $Q = P_{0,n'}$; we only need this case if $n\ge 2$.  The assumption $\Gmw\not\in P^*$ means that $\Gmw \not\in P_{0,n''}^*$ for any $n''>n'$.  Let us describe what these assumptions imply.  First, let $C$ denote the core of $\Gmw$, as defined in \S\ref{subsec:J-g-n}.  Then $G-E(C)$ is a disjoint union of trees $\{Y_v\}_{v\in V(C)}$.  Say that a core vertex $v \in V(C)$ {\em supports} a marked point $\alpha\in\{1,\ldots,n\}$ if $m(\alpha)\in Y_v$.  Then observe that for any $\Gmw\in \J_{g,n}$, the following are equivalent:
\begin{enumerate}[(1)]
\item $\Gmw \not \in P_{0,n''}^*$ for any $n''>n'$;
\item every core vertex of $\Gmw$ supports at most $n'$ markings.
\end{enumerate}
Now, we are assuming that $\Gmw$ satisfies (1), so it satisfies (2). Moreover (2) is evidently preserved by contracting $(0,n')$-bridges, since those operations never increase the number of markings supported by a core vertex.  So (1) is also preserved by contracting $(0,n')$-bridges, which is what we wanted to show.

Second, assume $Q=P_{1,0}$.  If $n\le 1$ then $P=\emptyset$ and we are done.  Otherwise, $P= P_{0,n}\cup \cdots \cup P_{0,2}$, and a graph $\Gmw$ is in $P^*$ if and only if $\Gmw$ has repeated markings.  The property $P$ is evidently preserved by uncontracting $(1,0)$-bridges, so we are done.

Third, assume $Q=P_{g',n'}$ with $g'\ge 1$  and $(g',n')\ne (1,0)$.  
Let $e\in E(G)$ be a $(g',n')$-bridge; we assume that $\Gmw \not \in P^*$ and we wish to show that $\Gmw/e \not \in P^*$.

First, in the case $n\ge 2$, we need to show that $\Gmw/e \not \in P_{0,n''}^*$ for any $n''$, i.e.,~$\Gmw/e$ has no repeated markings.  Since $\Gmw$ has no repeated markings, it suffices to show that not both ends of $e=v_1v_2$ are marked.  We may assume that the edge $(v_1,e)$ in $\Bl(\Gmw)$ was labeled $(g\!-\!g',n\!-\!n')$; we will show $v_1$ is unmarked.   Since $\Gmw \not \in P_{0,n''}^*$ for any $n''$ and $\Gmw\not \in P_{1,0}^*$, we have $w(v_1)=0$ and $v_1$ is at most once-marked. Therefore, there is at least one other block $B\ne e$ at $v_1$ and $(v_1,B)\in E(\Bl(\Gmw))$ is labelled $(>\!g',\ast)$ or $(g',\ge\! n')$.  In light of Equation~\eqref{eq:sum-around-art}, the only possibility is that there is only one such block $B$, $(v_1,B)$ is labelled $(g',n')$, and $v_1$ is unmarked.  Therefore $\Gmw/e$ has no repeated markings.  

Next, by Lemma~\ref{lem:art}\eqref{it:actually-spec}, every label $(g'',n'')$ on $E(\Bl(\Gmw))$ satisfies either $g''>g$, or $g''=g'$ and $n''\ge n'$.  Furthermore, the labels on $E(\Bl(\Gmw/e))$ are a subset of those on $E(\Bl(\Gmw))$.  Therefore by Lemma~\ref{lem:art}\eqref{it:converse}, $\Gmw/e \not \in P_{g'',n''}^*$ for any $g''<g'$ or $g''=g$ and $n''<n'$, as long as $g''\ge 1$.  We have verified that condition~\eqref{it:(2)} of Proposition~\ref{prop:prop} holds in the required cases.

To treat the last condition, regarding strictly co-$Q$ faces being in $P^*$, we assume all edges of $\Gmw\in \J_{g,n}$ are $(g',n')$-bridges. Then $G$ must be a tree with a single non-leaf vertex $v$, while every other vertex $v'$ has $w(v')=g'$ and $|m^{-1}(v)| = n'$.  Now we treat the following cases.

Suppose $Q=P_{g',n'}$ with $g'\ge 1$ and $(g',n')\ne (1,0)$. If $\Gmw$ has only $(g',n')$-edges then $\Gmw\in P_{1,0}^*$, since $\Gmw$ has positive weights. Therefore $\Gmw \in P^*$.

Next, suppose $Q=P_{1,0}$. If $\Gmw$ has only $(1,0)$-edges, then either $n\le 1$ and there is nothing to check, or $n\ge 2$ and so $v$ supports $n\ge 2$ markings.  Then $\Gmw \in P_{0,2}^*$,  so $\Gmw\in P^*$.  

Finally, suppose $n\ge 2$ and $Q=P_{0,n'}$ for $n'<n$. If $\Gmw$ has only $(0,n')$-edges for some $n'<n$, note that $w(v)=g$ and so $v$ may be expanded into a $(g,0)$-bridge, equivalently a $(0,n)$-bridge. So $\Gmw\in P_{0,n}^*$, and hence $\Gmw \in P^*$, as required.
\end{proof}

To close this section, we record a related contractibility result that will be useful for future applications.
Let $\Delta_{g,n}^{\mathrm{w}}$ denote the subcomplex of $\Delta_{g,n}$ parametrizing tropical curves that have at least one positive vertex weight.

\begin{lemma}\label{lem:x-contractible}  For all $g>0$, $\Delta_{g,n}^{\mathrm{w}}\cup  \Delta_{g,n}^{\mathrm{rep}}$ is contractible, unless $(g,n)=(1,1)$ in which case it is empty.
\end{lemma}
\begin{proof}
Regard $X=\Delta_{g,n}^{\text{w}}\cup  \Delta_{g,n}^{\text{rep}}$ as a symmetric $\Delta$-complex. In the above notation, the properties
$$P_{0,n},\ldots,P_{0,2},P_{1,0}$$
satisfy the hypotheses of Corollary~\ref{cor:sequence}.  The conclusion is that $X$ admits a strong deformation retract to the point in $X$ corresponding to the $1$-edge graph $\mathbf{B}(g,0)$. 
\end{proof}

\section{Calculations on $\Dgn$}
\label{sec:proofs}

In \S\ref{sec:g=1}, we apply the contractibility of $\Delta^\mathrm{rep}_{g,n}$ to  calculate the $S_n$-equivariant homotopy type of $\Delta_{1,n}$ and prove Theorem~\ref{thm:boundary} from the introduction.  In \S\ref{sec:proof-split} we prove Theorem~\ref{thm:split}.  In \S\ref{sec:transfer} we construct a transfer homomorphism from the cellular chain complex of $\Delta_g$ to that of $\Delta_{g,1}$.  Calculations of the rational homology of $\Delta_{g,n}$ in a range of cases for $g\ge 2$ are in Appendix~\ref{sec:compute}.

\subsection{The case $g=1$}\label{sec:g=1}

We  now restate and prove Theorem~\ref{thm:boundary}, showing that contracting $\Delta^{\mathrm{rep}}_{1,n}$ produces a bouquet of $(n-1)!/2$ spheres indexed by cyclic orderings of the set $\{1,\ldots,n\}$, up to order reversal, and then computing the representation of $S_n$ on the reduced homology of this bouquet of spheres induced by permuting the marked points.

\boundary*

\begin{proof} Recall that the core of a weighted, marked graph is the smallest connected subgraph containing all cycles and all vertices of positive weight.  The core of a genus $1$ tropical curve is either a single vertex of weight 1 or a cycle.  If $\Gamma \in \Dn \smallsetminus \Delta^{\mathrm{rep}}_{1,n}$ then the core of $\Gamma$ cannot be a vertex of weight 1, since then the underlying graph would be a tree, whose leaves would support repeated markings. Therefore the core of $\Gamma$ is a cycle with all vertices of weight zero. Since $\Gamma \not\in \Delta^{\mathrm{rep}}_{1,n}$, each vertex supports at most one marked point.  The stability condition then ensures that each vertex supports exactly one marked point.  In other words, the combinatorial types of tropical curves that appear outside the repeated marking locus consist of an $n$-cycle with the markings $\{1,\ldots,n\}$ appearing around that cycle in a specified order.  There are $(n-1)!/2$ possible orders $\tau$ of $\{1,\ldots,n\}$ up to symmetry, so we have $(n-1)!/2$ such combinatorial types $\Gmw_\tau.$

For $n \geq 3 $, each $\Gmw_\tau$ has no  nontrivial automorphisms, so the image of the interior of $\sigma^1(\Gmw_\tau)$ in $\Dn$ is an $(n-1)$-disc whose boundary is in $\Delta^{\mathrm{rep}}_{1,n}$.  Now it follows from Theorem~\ref{thm:contractible} that $\Dn$ has the homotopy type of a wedge of $(n-1)!/2$ spheres of dimension $n-1$.

It remains to identify the representation of $S_n$ on
$$V = H_{n-1}(\Dn/\Delta_{1,n}^{\mathrm{rep}}; \Q)$$
obtained by permuting the marked points.  We have already shown that $V$ has a basis given by the homology classes of the $(n-1)$-spheres in the wedge $\Dn/\Delta_{1,n}^{\mathrm{rep}}$, which are in bijection with the $(n-1)!/2$ unoriented cyclic orderings of $\{1,\ldots,n\}$.
Let $\phi\colon D_n \rightarrow S_n$ be the embedding of the dihedral group as a subgroup of the permutations of the vertices $\{1,\ldots,n\}$ of an $n$-cycle. Choose left coset representatives $\sigma_1,\ldots,\sigma_k$, where $k={(n-1)!/2}$, and write $[\sigma_i]$ for the corresponding basis elements of $V$.  For any $\pi \in S_n$, we have $\pi \sigma_i = \sigma_j \pi'$ for some $\pi'\in D_n$. Then $\pi\cdot[\sigma_i] = \pm [\sigma_j]$, where the sign depends exactly on the sign of the permutation on the {\em edges} of the $n$-cycle induced by $\pi'$. This is because the ordering of the edges determines the orientation of the corresponding sphere in $\Dn/\Delta_{1,n}^{\mathrm{rep}}$.  Therefore the representation of $S_n$ on $V$ is exactly $\mathrm{Ind}_{D_n,\phi}^{S_n}\mathrm{Res}_{D_n,\psi}^{S_n} \mathrm{sgn}$, where the restriction is according to the embedding of $\psi\colon D_n\rightarrow S_n$ into the group of permutations of edges of the $n$-cycle.
\end{proof}

\begin{remark}\label{rem:11-12-contractible}
We remark that $\Delta_{1,1}$ and $\Delta_{1,2}$ are contractible.  Indeed,  $\Delta_{1,1}$ is a point. And the unique cell of $\Delta_{1,2}$ not in $\Delta^{\mathrm{rep}}_{1,2}$ consists of two vertices and two edges between them.  Exchanging the edges gives a nontrivial $\Z/2\Z$ automorphism on this cell, which then retracts to $\Delta^{\mathrm{rep}}_{1,2}$.  So $\Delta_{1,2}$ is contractible by Theorem~\ref{thm:contractible}.  See Figure~\ref{fig:delta12}\end{remark}

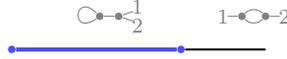
\begin{figure}[h!]
\begin{tikzpicture}[my_node/.style={fill, circle, inner sep=1pt, gray}, scale=.9]
\node[my_node, blue!70] (L) at (-2,0){};
\node[my_node, blue!70] (M) at (0.5,0){};
\draw[ultra thick, blue!70] (M) -- (L);
\draw[thick, black] (M) -- (1.75, 0);
\begin{scope}[scale=0.7, shift={(0,.2)}]
\node[my_node] (V2) at (-1,.5){};
\node[my_node] (V3) at (-.6,.5){};
\draw[gray] (V2) to [out=215, in = 135, looseness=15] (V2);
\draw[gray] (V2) to (V3);
\draw[gray] (V3) to (-.3,.6);
\draw[gray] (V3) to (-.3,.4);
\node[gray] at (-.2,.7){$\scriptstyle 1$};
\node[gray] at (-.2,.3){$\scriptstyle 2$};
\end{scope}
\begin{scope}[scale=0.7, shift = {(2,.7)}]
\node[my_node] (W1) at (0,0){};
\node[my_node] (W2) at (0.5,0){};
\draw[gray] (W1) to [bend right=45] (W2);
\draw[gray] (W1) to [bend left=45] (W2);
\draw[gray] (W1) to (-.3,0);
\draw[gray] (W2) to (.8,0);
\node[gray] at (-.4,0){$\scriptstyle 1$};
\node[gray]  at (.9,0){$\scriptstyle 2$};
\end{scope}
\end{tikzpicture}
\caption{$\Delta_{1,2}$, shown with two maximal symmetric orbits of simplices, retracts onto the subcomplex $\Delta_{1,2}^\mathrm{rep}$, shown in blue.} 
\label{fig:delta12}
\end{figure}

\subsection{Proof of Theorem~\ref{thm:split}}\label{sec:proof-split}

The results of the previous section allow us to prove Theorem~\ref{thm:split} from the introduction, restated below.  

\split*

\begin{proof}
Consider the cellular chain complex $ C_*(\Delta_{g,n},\Q)$. It is generated in degree $p$ by $[\Gmw,\omega]$ where $\Gmw\in \J_{g,n}$ and $\omega\colon E(\Gmw)\to [p]=\{0,1,\ldots,p\}$ is a bijection.  These generators are subject to the relations $[\Gmw,\omega]=\rm{sgn}(\sigma)[\Gmw',\omega']$ if there is an isomorphism $\Gmw\to\Gmw'$ inducing the permutation $\sigma$ of the set $[p]$.

Let $B^{(g,n)}$ be the subcomplex of $ C_*(\Delta_{g,n},\Q)$ spanned by the generators $[\Gmw,\omega]$ with at least one nonzero vertex weight.  Note that $B^{(g,n)}$ is in fact a subcomplex, since one-edge contractions of graphs with positive vertex weights have positive vertex weights. 

Define $A^{(g,n)}$ by the short exact sequence
$$0\to B^{(g,n)} \to  C_*(\Delta_{g,n},\Q) \to A^{(g,n)} \to 0.$$
Then $A^{(g,n)}$ is isomorphic to the marked graph complex $G^{(g,n)}$, up to shifting degrees by $2g-1$: a graph with $e$ edges is in degree $e-1$ in $A^{(g,n)}$ and in degree $e-2g$ in $G^{(g,n)}$.
And $B^{(g,n)}$ is the cellular chain complex associated to $\weight$, which is contractible whenever it is nonempty by Theorem~\ref{thm:contractible}.  Therefore, when $\weight$ is nonempty then $B^{(g,n)}$  is an acyclic complex, and the theorem follows.
\end{proof}

\begin{remark}
  The $n=0$ case is proved in \cite[\S4]{CGP1-JAMS}. The proof here is analogous, but carried out on the level of spaces.
\end{remark}

The proof of Theorem~\ref{thm:split} relied on the contractibility of $\weight$.  Using other natural contractible subcomplexes in place of $\weight$ would produce analogous results.  We pause to record a particular version which will be useful for  applications in \cite{cgp-chi}.

Let $K^{(g,n)}$ denote the following variant the marked graph complex $\Ggn$ from \S\ref{ssec:Ggn}.  As a graded vector space, it has generators $[\Gamma,\omega, m]$ for each connected graph $\Gamma$ of genus $g$ (Euler characteristic $1-g$) with or without loops, equipped with a total order $\omega$ on its set of edges and an {\em injective} marking function $m \colon \{ 1, \ldots, n \} \rightarrow V(G)$, such that the valence of each vertex plus the size of its preimage under $m$ is at least 3.  These generators are subject to the relations $$[\Gamma,\omega, m] = \mathrm{sgn}(\sigma) [\Gamma',\omega', m']$$ if there exists an isomorphism of graphs $\Gamma \cong \Gamma'$ that identifies $m$ with $m'$, and under which the edge orderings $\omega$ and $\omega'$ are related by the permutation $\sigma$.  The homological degree of $[\Gamma,\omega, m]$ is $e-2g$.  The differential on $K^{(g,n)}$ of $[\Gamma,\omega,m]$ is defined as before~\eqref{eq:the-differential}, with the added convention that if $e_i$ is a loop edge of $\Gamma$ then we interpret $[\Gamma/e, \omega|_{E(\Gamma)\smallsetminus\{e\}},\pi_i\circ m]$ as $0$.

\begin{proposition}
Fix $g>0$ and $n\ge 0$ with $2g-2+n>0$, excluding $(g,n)=(1,1)$.  For all $k$, we have isomorphisms on homology $$H_k (K^{(g,n)}) \xrightarrow{\cong} \widetilde H_{k+2g-1}(\Dgn;\Q).$$
\end{proposition}

\begin{proof}
The complex $K^{(g,n)}$ is isomorphic, after shifting degrees by $2g-1$, to the relative cellular chain complex
$ C_*(\Dgn,\Dgn^{\mathrm{w}}\cup \rep;\Q)$ of the pair of symmetric $\Delta$-complexes $\Dgn^{\mathrm{w}}\cup \rep \subset \Dgn$, as defined in \S\ref{sec:relative-homology}.  But $\Dgn^{\mathrm{w}}\cup \rep$ is contractible by Lemma~\ref{lem:x-contractible}, so we have identifications
\begin{equation*}
  H_k (K^{(g,n)}) \cong H_{k+2g-1}( C_*(\Delta_{g,n}, \Dgn^{\mathrm{w}}\cup \rep; \Q)) \cong \widetilde H_{k+2g-1} (\Delta_{g,n};\Q).
  \qedhere
\end{equation*}
\end{proof}

\subsection{A cellular transfer map}  \label{sec:transfer}
In \cite{CGP1-JAMS}, we showed that $\bigoplus_g H_*(\Delta_g;\Q)$ is large and has a rich structure; its dual contains the Grothendieck-Teichm\"uller Lie algebra $\mathfrak{grt}_1$, and $\dim_\Q H_{*}(\Delta_g;\Q)$ grows at least exponentially with $g$.  Here we restate and prove Theorem~\ref{thm:transfer}, showing that nontrivial homology classes on $\Delta_g$ give rise to nontrivial classes with a marked point.  

\transfer*

\begin{proof}
We begin by defining the map $t \colon   C_*(\Dg;\Q) \rightarrow  C_*(\Delta_{g,1};\Q)$.  For each vertex $v$ in a stable, vertex weighted graph $\Gmw \in \J_g$, let $\chi(v) = 2 w(v) - 2 + \val(v)$.  Note that, for a vertex weighted graph $\Gmw$ of genus $g$, we have $\sum_{v \in V(\Gmw)} \chi(v) = 2g-2$.

Now, consider an element $[\Gmw, \omega]$ of $\Delta_g([p])$, i.e., the isomorphism class of a pair $(\Gmw,\omega)$, where $\Gmw$ is a stable graph of genus $g$ and $\omega$ is an ordering of its $p+1$ edges.  For each vertex $v$, let $[\Gmw_v, \omega] \in \Delta_{g,1}([p])$ be the stable marked graph with ordered edges obtained by marking the vertex $v$.  The linear map $\Q \Delta_g([p]) \rightarrow \Q \Delta_{g,1}([p])$ given by
\[
[\Gmw, \omega] \mapsto \sum_v \chi(v) [\Gmw_v, \omega]
\]
commutes with the action of $S_{p+1}$, and hence, after tensoring with $\Q^{\mathrm{sgn}}$ and taking coinvariants, induces a map $t_p \colon   C_p(\Dg;\Q) \rightarrow  C_p(\Delta_{g,1};\Q)$.  

We claim that $t = \bigoplus_p t_p$ commutes with the differentials on $ C_*(\Dg;\Q)$ and $ C_*(\Delta_{g,1};\Q)$.  Recall that each differential is obtained as a signed sum over contractions of edges.  The claim then follows from the observation that, if $v$ is the vertex obtained by contracting an edge with endpoints $v'$ and $v''$, then $\chi(v) = \chi(v') + \chi(v'')$.  This shows that $t$ is a map of chain complexes.  Furthermore, by construction, $t$ maps graphs without loops or vertices of positive weight to marked graphs without loops or vertices of positive weight, and hence takes the subcomplex $G^{(g)}$ into $G^{(g,1)}$.  

It remains to show that these maps of chain complexes induce injections $\widetilde H_k(\Dg;\Q) \hookrightarrow \widetilde H_k(\Delta_{g,1};\Q)$ and $H_k(G^{(g)}) \hookrightarrow H_k(G^{(g,1)})$, for all $k$.  To see this, we construct a map $\pi\colon  C_*(\Delta_{g,1};\Q)\rightarrow  C_*(\Delta_{g};\Q) $ such that $\pi \circ t$ is multiplication by $2g-2$.

Let $\widetilde \pi \colon  \Q \Delta_{g,1}([p]) \rightarrow \Q \Delta_{g}([p])$ be the linear map obtained by forgetting the marked point.  More precisely, if forgetting the marked point on $\Gmw \in \J_{g,1}$ yields a stable graph $\Gmw_0 \in \J_g$, then $\widetilde \pi$ maps $[\Gmw, \omega]$ to $[\Gmw_0, \omega]$.  If forgetting the marked point on $\Gmw$ yields an unstable graph, then $\widetilde \pi$ maps $[\Gmw, \omega]$ to 0.  (Forgetting the marked point yields an unstable graph exactly when the marking is carried by a weight zero vertex incident to exactly two half-edges.)  The resulting linear map $\widetilde \pi$ commutes with the action of $S_{p+1}$, so tensoring with $\Q^{\mathrm{sgn}}$ and taking coinvariants gives $\pi_p  \colon  C_p(\Delta_{g,1};\Q) \rightarrow  C_p(\Dg;\Q)$.  Let $\pi = \bigoplus_p \pi_p$.  One then checks directly that $\pi \circ t$ is multiplication by $2g-2$, and that $\pi$ commutes with the differentials.

The only subtlety to check is as follows. Suppose $[\Gmw,\omega] \in \Delta_{g,1}([p])$ is such that forgetting the marked point results in an unstable graph.  Then the vertex supporting the marked point is incident to exactly two edges $e,e'$.  
Then in the expression
$$\partial [\Gmw,\omega] = \sum_{i=0}^p (-1)^i [\Gmw/e_i, \omega/e_i]$$
in all but exactly two terms $(-1)^i [\Gmw/e_i, \omega/e_i]$, forgetting the marked point results in an unstable graph.  The two exceptional terms correspond to the two edges $e,e'$, and these cancel under $\pi$.
\end{proof}

\begin{corollary}\label{cor:hdg1} We have
$$\dim H_{2g-1}(\Delta_{g,1};\Q) > \beta^g + \text{constant}$$
for any $\beta < \beta_0 \approx 1.32\ldots$, where $\beta_0$ is the real root of $t^3-t-1=0.$
\end{corollary}
\begin{proof}
  The analogous result for $\Delta_g$ is proved in \cite{CGP1-JAMS}; now combine with Theorem~\ref{thm:transfer}.
\end{proof}

\begin{remark}
  As mentioned in the introduction, the splitting on the level of cohomology was constructed earlier in \cite{TurchinWillwacher17}.  The splitting they construct is induced by Lie bracket with a graph $L$ which is a single edge between two vertices, tracing through definitions, this is (at least up to signs and grading conventions) dual to the restriction of our $t$ to a chain map $G^{(g)} \to G^{(g,1)}$.
\end{remark}

\begin{remark}
  For all $n$, there is a natural map $M_{g,n+1}^\trop \rightarrow M_{g,n}^\trop$ obtained by forgetting the marked point and stabilizing \cite{acp}.  When $n = 0$, the preimage of $\bullet_g$ is $\bullet_{g,1}$, so there is an induced map on the link $\Delta_{g,1} \rightarrow \Dg$.  This continuous map of topological spaces does not come from a map of 
  symmetric $\Delta$-complexes
  (because some cells of $\Delta_{g,1}$ are mapped to cells of lower dimension in $\Dg$), but one can check that the pushforward on rational homology is induced by $\pi$.  When $n > 1$, the preimage of $\bullet_{g,n}$ includes graphs other than $\bullet_{g,n+1}$, and there is no induced map from $\Delta_{g,n+1}$ to $\Dgn$. 
\end{remark}

\section{Applications to, and from, $\cM_{g,n}$} \label{sec:applications}

\subsection{The boundary complex of $\cM_{g,n}$} \label{sec:dual-compl-simple}

We recall that the dual complex $\Delta(D)$ of a normal crossings divisor $D$ in a smooth, separated Deligne--Mumford (DM) stack $X$ is naturally defined as a symmetric $\Delta$-complex \cite[\S5.2]{CGP1-JAMS}.  Over $\mathbb{C}$, for each $p \geq -1$, $\Delta(D)_p$ is the set of equivalence classes of pairs $(x,\sigma)$, where $x$ is a point in a stratum of codimension $p$ in $D$ and $\sigma$ is an ordering of the $p+1$ analytic branches of $D$ that meet at $x$.  The equivalence relation is generated by paths within strata: if there is a path from $x$ to $x'$ within the codimension $p$ stratum and a continuous assignment of orderings of branches along the path, starting at $(x,\sigma)$ and ending at $(x', \sigma')$, then we set $(x, \sigma) \sim (x',\sigma')$.  

This dual complex can equivalently be defined (and, more generally, over fields other than $\C$), using normalization and iterated fiber product.   Let $\EE \to D$ be the normalization of $D$ and write $$\EEE{p} = \EE \times_X \dots \times_X \EE$$ for the $(p+1)$-fold
iterated fiber product.  Define $\FFF{p} \subset \EEE{p}$ as the open subvariety consisting of $(p+1)$-tuples of pairwise distinct points in $\EEE{p}$ that all lie over the same point of $D$.  We can then define $\Delta(D)_p$ to be the set of irreducible components of $\FFF{p}$.
(Note that, over $\C$, a point of $\FFF{p}$ encodes exactly the same data as a point in the codimension $p$ stratum together with an ordering of the $p+1$ analytic branches of $D$ at that point.)   For further details, see \cite[\S5]{CGP1-JAMS}.

If $X$ is proper then the simple homotopy type of this dual complex depends only on the open complement $X \smallsetminus D$ \cite{boundarycx, Harper17}, and its reduced rational homology is naturally identified with the top weight cohomology of $X \smallsetminus D$.  More precisely, if $X$ has pure dimension $d$, then
\begin{equation}\label{eq:topweight}
\widetilde H_{k-1}(\Delta(D); \Q) \cong \Gr_{2d}^W H^{2d-k} (X\smallsetminus D;\Q).
\end{equation}
See \cite[Theorem~5.8]{CGP1-JAMS}.  

\medskip

Most important for our purposes is the special case where $X = \ocM_{g,n}$ is the Deligne--Mumford stable curves compactification of $\cM_{g,n}$ and $D = \ocM_{g,n} \smallsetminus \cM_{g,n}$ is the boundary divisor.

\begin{theorem}\label{cor_dual_complex}
  The dual complex of the boundary divisor in the moduli space of
  stable curves with marked points
  $\Delta(\ocM_{g,n} \smallsetminus \cM_{g,n})$ is $\Dgn$.
\end{theorem}

\begin{proof}
Modulo the translation between smooth generalized cone complexes and symmetric $\Delta$-complexes, this is one of the main results of \cite{acp}, to which we refer for a thorough treatment.  The details of the construction for $n = 0$ are also explained in \cite[Corollaries~5.6 and 5.7]{CGP1-JAMS}, and the general case is similar.
\end{proof}

As an immediate consequence of Theorem~\ref{cor_dual_complex} and \eqref{eq:topweight}, the reduced rational homology of $\Dgn$ agrees with the top weight cohomology of $\cM_{g,n}$. 

\begin{corollary} \label{cor:mgn-top-wt}
There is a natural isomorphism
\[
\Gr_{6g-6+2n}^W H^{6g-6+2n-k} (\cM_{g,n}; \Q) \xrightarrow{\sim} \widetilde{H}_{k-1}(\Dgn;\Q) ,
\]
identifying the reduced rational homology of $\Dgn$ with the top graded piece of the weight filtration on the cohomology of $\cM_{g,n}$.
\end{corollary}  

In the case $g=1$, we have a complete understanding of the rational homology of $\Delta_{1,n}$, from Theorem~\ref{thm:boundary}. Thus we immediately deduce a similarly complete understanding of the top weight cohomology of $\cM_{1,n}$, stated as Corollary~\ref{thm:genus1} in the introduction.

\begin{corollary*}  
The top weight cohomology of $\cM_{1,n}$ is supported in degree $n$, with rank $(n-1)!/2$, for $n \geq 3$.  Moreover, the representation of $S_n$ on $\Gr_{2n}^W H_{n}(\cM_{1,n}; \Q)$ induced by permuting marked points is
\[
\mathrm{Ind}_{D_n,\phi}^{S_n} \, \mathrm{Res}^{S_n}_{D_n,\psi} \, \mathrm{sgn}.
\]
\end{corollary*}

\begin{remark}\label{rem:M1n-stuff}
The fact that the top weight cohomology of $\cM_{1,n}$ is supported in degree $n$ can also be seen without tropical methods, as follows.  The rational cohomology of a smooth Deligne--Mumford stack agrees with that of its coarse moduli space, and the coarse space $M_{1,n}$ is affine.  To see this, note that $M_{1,1}$ is affine, and the forgetful map $M_{g,n+1} \rightarrow M_{g,n}$ is an affine morphism for $n \geq 1$.  It follows that $M_{1,n}$ has the homotopy type of an $n$-dimensional CW-complex, by \cite{AndreottiFrankel59, Karcjauskas77}, and hence $H^*(\cM_{1,n}; \Q)$ is supported in degrees less than or equal to $n$. The weights on $H^k$ are always between $0$ and $2k$, so the top weight $2n$ can appear only in degree $n$.  
\end{remark}

\begin{remark}\label{rem:more-M1n-stuff}
  Getzler has calculated an expression for the $S_n$-equivariant Serre characteristic of $\cM_{1,n}$ \cite[(5.6)]{Getzler99}.  Since the top weight cohomology is supported in a single degree, it is determined as a representation by this equivariant Serre characteristic.  We do not know how to deduce Corollary~\ref{thm:genus1} directly from Getzler's formula.  However, C.~Faber has shown a formula for $\Gr_{2n}^W H_n(\cM_{1,n};\Q)$, as an $S_n$-representation, that is derived from \cite[Theorem 2.5]{Getzler98}.  See \cite[Theorem 1.5]{cgp-chi}.
\end{remark}

\begin{remark} \label{rem:even-more-M1n-stuff}
Petersen explains that it is possible to adapt the methods from \cite{Petersen14} to recover the fact that the top weight cohomology of $\cM_{1,n}$ has rank $(n-1)!/2$, using the Leray spectral sequence for $\cM_{1,n} \rightarrow \cM_{1,1}$ and the Eichler--Shimura isomorphism \cite{Petersen15}.
\end{remark}

Using Corollary~\ref{cor:mgn-top-wt} and the transfer homomorphism from \S\ref{sec:transfer}, we also deduce an exponential growth result for top-weight cohomology of $\cM_{g,1}$, stated as Corollary~\ref{cor:mgoneexp} in the introduction.

\mgoneexp*

\begin{proof}
This follows from Corollary~\ref{cor:hdg1} and Corollary~\ref{cor:mgn-top-wt}.
\end{proof}

\begin{remark}\label{rem:mg1} 
Corollary~\ref{cor:mgoneexp} above, and the existence of a natural injection $H_k(\Delta_{g};\Q) \rightarrow H_k(\Delta_{g,1};\Q)$, 
may also be deduced purely algebro-geometrically.  Indeed, pulling back along the forgetful map $\cM_{g,1} \rightarrow \cM_g$ and composing with cup product with the Euler class is injective on rational singular cohomology.  This is because further composing with the Gysin map (proper push-forward) induces multiplication by $2g-2$ on $H^*(\cM_g;\Q)$.  Furthermore, this injection maps top weight cohomology into top weight cohomology, because cup product with the Euler class increases weight by 2.  Identifying top weight cohomology of $\cM_{g,n}$ with rational homology of $\Delta_{g,n}$ for $n = 0$ and $n = 1$ then gives a natural injection $H_k(\Delta_{g};\Q) \rightarrow H_k(\Delta_{g,1};\Q)$, as claimed.  Presumably, this map agrees with the one defined in \S\ref{sec:transfer} up to some normalization constant.
\end{remark}

The following is a strengthening of Corollary~\ref{cor:zero-section-complement}.
\begin{corollary}\label{cor:1.9-strengthened}
  Let $\mathrm{Mod}_g^1$ denote the mapping class group of a connected oriented 2-manifold of genus $g$, with one marked point, and let $G$ be any group fitting into an extension
  \begin{equation*}
    \Z \to G \to \mathrm{Mod}_g^1,
  \end{equation*}
  Then
  \begin{equation*}
    \dim H^{4g-3}(G;\Q) > \beta^g+\text{constant}
  \end{equation*}
  for any $\beta < \beta_0 \approx 1.32\ldots$, where $\beta_0$ is the real root of $t^3-t-1=0$.

  In particular, $G = \mathrm{Mod}_{g,1}$, the mapping class group with one parametrized boundary component, satisfies this dimension bound on its cohomology.
\end{corollary}
\begin{proof}
  Let us write $B\mathrm{Mod}_g^1$ for the classifying space of the discrete group $\mathrm{Mod}_g^1$, i.e., a $K(\pi,1)$ for this group.  Its homology is the group homology of $\mathrm{Mod}_g^1$, and its rational cohomology is canonically isomorphic to $H^*(\cM_{g,1};\Q)$.  In particular $\dim H^{4g-4}(B\mathrm{Mod}_g^1;\Q)$ is bounded below by Corollary~\ref{cor:mgoneexp}.
  
  The extension is classified by a class $e \in H^2(\mathrm{Mod}_g^1;\Z)$, which is the first Chern class of some principal $U(1)$-bundle $\pi: P \to B\mathrm{Mod}_g^1$, and we have $P \simeq BG$.  Part of the Gysin sequence for $\pi$ looks like
  \begin{equation*}
    \dots \xrightarrow{e \cup} H^{4g-3}(B\mathrm{Mod}_g^1) \xrightarrow{\pi^*} H^{4g-3}(BG) \xrightarrow{\pi_*} H^{4g-4}(B\mathrm{Mod}_g^1) \xrightarrow{e \cup}
    H^{4g-2}(B\mathrm{Mod}_g^1) \xrightarrow{\pi^*} \dots,
  \end{equation*}
  and by Harer's theorem \cite{Harer86} that $\mathrm{Mod}_g^1$ is a virtual duality group of virtual cohomological dimension $4g-3$ we get $H^{4g-2}(B\mathrm{Mod}_g^1;\Q) = 0$, and hence a surjection $\pi_*: H^{4g-3}(BG) \to H^{4g-4}(B\mathrm{Mod}_g^1)$.

  It is well-known that
  $G = \mathrm{Mod}_{g,1}$ fits into such an extension, where the $\Z$ is generated by Dehn twist along a boundary-parallel curve (see e.g.,\ \cite[Proposition 3.19]{FarbMargalit12}).
\end{proof}

\begin{remark}
  By methods similar to those outlined in \cite[\S6]{CGP1-JAMS}, the dimension bounds in the Corollaries above may be upgraded to explicit injections of graded vector spaces, from the Grothendieck--Teichm\"uller Lie algebra to $\prod_{g \geq 3} H_{4g-4}(\cM_{g,1};\Q)$ and $\prod_{g \geq 3} H_{4g-3}(\mathrm{Mod}_{g,1};\Q)$, respectively.
\end{remark}

\subsection{Support of the rational homology of $\Delta_{g,n}$} \label{sec:support}

In \cite{CGP1-JAMS}, we observed that known vanishing results for the cohomology of $\cM_{g}$ imply that the reduced rational homology of $\Delta_g$ is supported in the top $g-2$ degrees, and that the homology of $G^{(g)}$ vanishes in negative degrees.  Here we prove the analogous result with marked points, stated as Theorem~\ref{thm:hgvanishing} in the introduction, using Harer's computation of the virtual cohomological dimension of $\cM_{g,n}$ from \cite{Harer86}.

\hgvanishing*

\begin{proof}
The case $n = 0$ is proved in \cite{CGP1-JAMS} and the case $g = 1$ follows from Theorem~\ref{thm:genus1}.  

Suppose $g \geq 2$ and $n \geq 1$.  By \cite{Harer86}, the virtual cohomological dimension of $\cM_{g,n}$ is $4g - 4 + n$.  Furthermore, when $n = 1$, we have $H^{4g-3}(\cM_{g,1};\Q) = 0$, by \cite{ChurchFarbPutman12}.   Therefore, the top weight cohomology of $\cM_{g,n}$ is supported in degrees less than $4g - 3 + n - \delta_{1,n},$ where $\delta_{i,j}$ is the Kronecker $\delta$-function.  By Corollary~\ref{cor:mgn-top-wt}, it follows that $\widetilde H_k(\Dgn; \Q)$ is supported in degrees less than $\max\{ 2g - 1, 2g + n - 3 \}$, as required.
\end{proof}

\noindent It would be interesting to have a proof of this vanishing result using the combinatorial topology of $\Delta_{g,n}$.

\section{Remarks on stability}

It is natural to ask whether the homology of $\Dgn$ can be related to
known instances of {\em homological stability} for the complex moduli
space of curves $\cM_{g,n}$ and for the free group $F_g$.  Here, we
comment briefly on the reasons that the tropical
moduli space $\Dgn$ relates to both $\cM_{g,n}$ and $F_g$.

{Homological stability} has been an important point of view in the
understanding of $\cM_{g,n}$; we are referring to the fact that the
cohomology group $H^k(\cM_{g,n}; \Q)$ is independent of $g$ as long as
$g \geq 3k/2 + 1$ \cite{Harer, Ivanov93, Boldsen12}.  The structure of
the rational cohomology in this stable range was famously conjectured by
Mumford, for $n = 0$, and proved by Madsen and Weiss \cite{MadsenWeiss07}; see \cite[Proposition~2.1]{Looijenga96} for the extension to $n > 0$.  There
are certain \emph{tautological classes}
$\kappa_i \in H^{2i}(\cM_{g,n})$ and $\psi_j \in H^2(\cM_{g,n})$, and
the induced map
\begin{equation*}
  \Q[\kappa_1, \kappa_2, \dots] \otimes \Q[\psi_1, \dots, \psi_n] \to H^*(\cM_{g,n};\Q) 
\end{equation*}
is an isomorphism in the ``stable range'' of degrees up to $2(g-1)/3$.

A similar homological stability phenomenon happens for
\emph{automorphisms of free groups}.  If $F_g$ denotes the free group on
$g$ generators and $\mathrm{Aut}(F_g)$ is its automorphism group, then
Hatcher and Vogtmann \cite{HatcherVogtmann98} proved that the group
cohomology $H^k(\mathrm{Aut}(F_g))$ is independent of $g$ as long
as $g \gg k$.  In \cite{Galatius11} it was proved that an analogue of the
Madsen--Weiss theorem holds for these groups: the rational cohomology
$H^k(B\mathrm{Aut}(F_g);\Q)$ vanishes for $g \gg k \geq 1$.

The tropical moduli space $\Dgn$ is closely related to \emph{both} of
these objects.  On the one hand its reduced rational homology is
identified with the top weight cohomology of $\cM_{g,n}$.  On the
other hand it is also closely related to $\mathrm{Aut}(F_g)$, as we
shall now briefly explain. 

\subsection{Relationship with automorphism groups of free groups}
\label{sec:relat-with-autom}

Let us follow the terminology of \cite{Caporaso13} and call a tropical curve \emph{pure} if all its vertices have zero weights.  Isomorphism classes of pure tropical curves are parametrized by an open subset
\begin{equation*}
  \Dgn^{\mathrm{pure}} = |\Dgn| \smallsetminus |\Dgn^w| \subset |\Dgn|.
\end{equation*}
Its points are isometry classes of triples $(G,m,w,\ell)$ with $(G,m,w) \in \J_{g,n}$, such that $w = 0$ and $\ell(e) > 0$ for all $e \in E(G)$.  These spaces are related to Culler and Vogtmann's ``outer space'' \cite{CullerVogtmann86} and its versions with marked points, e.g.,\ \cite{HatcherVogtmann98}.  Indeed, for $n=1$ for example, outer space $X_{g,1}$ can be regarded as the space of isometry classes of triples $(G,m,w,\ell,h)$ where $(G,m,w) \in \J_{g,1}$ are as before, with $w =0$ and $\ell(e) > 0$ for all $e$, and $h: F_g \to \pi_1(G,m(1))$ is a specified isomorphism between the free group $F_g$ on $g$ generators and the fundamental group of $G$ at the point $m(1) \in V(G)$.  The group $\Aut(F_g)$ acts on $X_{g,1}$ by changing $h$, and the forgetful map $(G,m,w,\ell,h) \mapsto (G,m,w,\ell)$ factors over a homeomorphism
\begin{equation*}
  X_{g,1} / \Aut(F_g) \xrightarrow{\approx} \Delta_{g,1}^\mathrm{pure}.
\end{equation*}
Since $X_{g,1}$ is contractible (\cite{CullerVogtmann86,HatcherVogtmann98}) and the stabilizer of any point in $X_{g,1}$ is finite, there is a map
\begin{equation*}
  B\mathrm{Aut}(F_g) \to \Delta_{g,1}^\mathrm{pure}
\end{equation*}
which induces an isomorphism in rational cohomology.
(Recall that $B\mathrm{Aut}(F_g)$ denotes the
\emph{classifying space} of the discrete group $\mathrm{Aut}(F_g)$.
It is a $K(\pi,1)$ space whose singular cohomology is isomorphic to
the group cohomology of $\mathrm{Aut}(F_g)$.)  More generally, there
are groups $\Gamma_{g,n}$
defined up to isomorphism by $\Gamma_{g,0} \cong \mathrm{Out}(F_g)$, and
$\Gamma_{g,n} = \mathrm{Aut}(F_g) \ltimes F_g^{n-1}$ for $n > 0$,
\cite{Hatcher95,
  HatcherVogtmann04}.  The groups $\Gamma_{g,n}$ are also isomorphic to the
groups denoted $A_{g,n}$ in \cite{HatcherWahl10}.

By a similar argument as above, which ultimately again rests on contractibility of outer space, the space $\Delta_{g,n}^\mathrm{pure}$ is a rational model for the group $\Gamma_{g,n}$, in the sense that there is a map
\begin{equation*}
  B\Gamma_{g,n} \to \Delta_{g,n}^\mathrm{pure}
\end{equation*}
inducing an isomorphism in rational homology.  A similar rational model for $B\Gamma_{g,n}$ was considered in \cite[\S6]{CHKV}, and may in fact be identified with a deformation retract of $\Delta_{g,n}^\mathrm{pure}$.  (We shall not need this last fact, but let us nevertheless point out that the subspace $Q_{g,n} \subset \Dgn^\mathrm{pure}$ defined as parametrizing stable
  tropical curves with zero vertex weights in which the $n$ marked
  points are on the {\em core}, as defined in \S\ref{subsec:J-g-n}, is a strong deformation retract of $\Delta_{g,n}^\mathrm{pure}$.  The deformation retraction is given by uniformly shrinking the non-core edges and
  lengthening the core edges, where the rate of lengthening of each
  core edge is proportional to its length.  This $Q_{g,n}$ is homeomorphic to the space considered in \cite[\S6]{CHKV} under the same notation.)

Therefore, the inclusion $\iota \colon  \Dgn^\mathrm{pure} \subset \Dgn$
induces a map in rational homology
\begin{equation}\label{eq:3}
  \widetilde H_*(\Gamma_{g,n}; \Q)\cong \widetilde H_*(\Dgn^\mathrm{pure};\Q)
  \xrightarrow{\iota_*} \widetilde H_*(\Dgn; \Q) \cong \Gr_{6g-6+2n}^W
  H^{6g-7+2n-*}(\cM_{g,n}; \Q).
\end{equation}
By compactness of $\Dgn$ and, for $g > 0$ and $2g + n > 3$, contractibility of $\Dgn^{\rm w}$, this map may equivalently be described as the linear dual of the canonical map from compactly supported cohomology to cohomology of $\Dgn^\mathrm{pure}$.  For $n = 0$ the map $\iota_*$ in particular gives a map 
$$\widetilde H_*(\mathrm{Out}(F_g);\Q) \to \Gr^W_{6g-6} H^{6g-7-*}(\cM_g;\Q).$$

It is intriguing to note, as emphasized to us by a referee, that group homology of $\Gamma_{g,n}$ is also calculated by a kind of graph complex, although different from $G^{(g,n)}$.  For $n=0$ for instance, this the ``Lie graph complex'' calculating group cohomology of $\Out(F_r)$ (see  \cite{Kontsevich93,Kontsevich94} and \cite[Proposition 21, Theorem 2]{ConantVogtmann03}).  In that complex, vertices of valence $m$ are labeled by elements of $\mathrm{Lie}(m-1)$, operations of arity $m-1$ in the Lie operad.  The complex $G^{(g)}$ and the Lie graph complex both have boundary homomorphism involving contraction of edges, but the Lie graph complex calculates cohomology $H^*(\mathrm{Out}(F_g);\Q) = H^*(\Delta_g^\mathrm{pure};\Q)$.  The dual of the Lie graph complex then calculates homology $H_*(\mathrm{Out}(F_g);\Q) = H_*(\Delta_g^\mathrm{pure};\Q)$, but the differential on this dual involves expanding vertices, instead of collapsing edges.  The homomorphism $\iota_*$ therefore seems a bit mysterious from this point of view, going from homology of a graph chain complex to cohomology of a graph cochain complex.  It would be interesting to understand this better on a chain/cochain level, but at the moment we have nothing substantial to say about it.

\subsection{(Non-)triviality of $\iota_*$}
\label{sec:non--triviality}

Known properties of $\cM_{g,n}$ and $\Delta_{g,n}$
severely limit the possible degrees in which~(\ref{eq:3}) may be
non-trivial.
Indeed, by Theorem~\ref{thm:hgvanishing}, the reduced homology of $\Dgn$ vanishes in degrees below $\max\{2g-1, 2g-3+n \}$.   On the other hand, $ H_*(\Gamma_{g,n}; \Q)$ is supported in
degrees at most $2g-3+n$ by \cite[Remark~4.2]{CHKV}.  It follows that
$\iota_*$ vanishes in all degrees except possibly $* = 2g-3 + n$, for $n > 0$, where it gives a homomorphism
\begin{equation*}
  H_{2g-3+n}(\Dgn^\mathrm{pure};\Q) \to H_{2g-3+n}(\Dgn; \Q).
\end{equation*}
In this degree, the homomorphism is \emph{not}
always trivial.  Indeed, for $g = 1$, the domain
$H_{n-1}(\Delta_{1,n}^\mathrm{pure};\Q)$ is one-dimensional and the
map into $H_{n-1}(\Delta_{1,n};\Q) \cong \Q^{(n-1)!/2}$ is injective.

\begin{proposition}\label{prop:CHKV-comparison}
  For $n \geq 3$ odd, the map
  $H_{n-1}(\Delta_{1,n}^\mathrm{pure};\Q) \to
  H_{n-1}(\Delta_{1,n};\Q)$ is nontrivial.
\end{proposition}

\begin{proof}[Proof sketch]
  The subspace $\Delta_{1,n}^\mathrm{pure} \subset \Delta_{1,n}$  is homotopy equivalent to the space $Q_{1,n}$, 
  which is the orbit space
  $((S^1)^n)/O(2)$, where $O(2)$ acts by rotating and reflecting all
  $S^1$ coordinates.  Moreover, $\Delta_{1,n}$ is homotopy equivalent to the
  orbit space $((S^1)^n / R)/O(2)$, where $R \subset (S^1)^n$ is the ``fat diagonal'' consisting of points where two coordinates agree.  $(S^1)^n/R$ denotes the quotient space obtained by collapsing $R$, and the homotopy equivalence $\Delta_{1,n} \simeq ((S^1)^n/R)/O(2)$ follows from the contractibility of the bridge locus.

  In this description, the inclusion of
  $\Delta_{1,n}^\mathrm{pure} \hookrightarrow \Delta_{1,n}$ is modeled
  by the obvious quotient map collapsing $R$ to a point.  We have the homeomorphism
  $(S^1)^n/O(2) = (S^1)^{n-1}/O(1)$ and the map
  $H_{n-1}(\Delta_{1,n}^\mathrm{pure};\Q) \to
  H_{n-1}(\Delta_{1,n};\Q)$
  becomes identified with the $O(1)$ coinvariants of the map
  $\Q \cong H_{n-1}((S^1)^{n-1};\Q) \to H_{n-1}((S^1)^n/R;\Q) \cong
  \Q^{(n-1)!}$,
  which sends the fundamental class of $(S^1)^{n-1}$ to the
  ``diagonal'' class, i.e.,\ the sum of all fundamental classes of
  $S^{n-1}$ in our description of $\Delta_{1,n}$ as a wedge of
  $(S^{n-1})$'s.
\end{proof}

\subsection{Stable homology}
\label{sec:stable-homology}
One of the initial motivations for this
paper was to use the tropical moduli space to provide a direct link
between moduli spaces of curves, automorphism groups of free groups,
and their homological stability properties.
In light of homological stability for $\Gamma_{g,n}$ and $\cM_{g,n}$,
it is natural to try to form some kind of direct limit of~$\Dgn$ as
$g \to \infty$.  For $n=1$, there is indeed a map
$\Delta_{g,1} \to \Delta_{g+1,1}$, which sends a tropical curve
$G \in \Delta_{g,1}$ to ``$G \vee S^1$''.  More precisely, the map
adds a single loop to $G$ at the marked point, and appropriately
normalizes edge lengths (for example, multiply all edge lengths in
$G$ by $\frac12$ and give the loop length $\frac12$).  This map
fits with the stabilization map for $B\mathrm{Aut}(F_g)$ into a
commutative diagram of spaces,
\begin{equation}\label{eq:2}
  \begin{aligned}
    \xymatrix{
      {B\mathrm{Aut}(F_g)} \ar[d] \ar[r]^-{\simeq_\Q} &
      {\Delta_{g,1}^\mathrm{pure}} \ar[d] \ar[r] &
      {\Delta_{g,1}} \ar[d]\\
      {B\mathrm{Aut}(F_{g+1})} \ar[r]_-{\simeq_\Q} &
      {\Delta_{g+1,1}^\mathrm{pure}}
      \ar[r] &{\Delta_{g+1,1}}.
    }
  \end{aligned}
\end{equation}
The leftmost vertical arrow is studied in \cite{HatcherVogtmann98}, where it is shown to induce an isomorphism in homology in degree up to $(g-3)/2$.  

For the outer automorphism group $\mathrm{Out}(F_g)$, there is a
similar comparison diagram
\begin{equation*}
  \xymatrix{
    {B\mathrm{Aut}(F_g)} \ar[d] \ar[r]^-{\simeq_\Q} &
    {\Delta_{g,1}^\mathrm{pure}} \ar[d] \ar[r] &
    {\Delta_{g,1}} \ar[d]\\
    {B\mathrm{Out}(F_{g})} \ar[r]_-{\simeq_\Q} &
    {\Delta_{g,0}^\mathrm{pure}}
    \ar[r] &{\Delta_{g,0}}.
  }
\end{equation*}

In light of this relationship between $\cM_{g,n}$ and
$\mathrm{Out}(F_g)$ and $\Delta_{g,n}$ and
$\Delta_{g,n}^\mathrm{pure}$, and in light of \cite{MadsenWeiss07} and
\cite{Galatius11}, it is tempting to ask about a limiting cohomology
of $\Delta_{g,1}$ as $g \to \infty$.  However, this limit seems to be
of a different nature from the corresponding limits for
$B\mathrm{Aut}(F_g)$ and $\cM_{g,n}$, as in the observation below.  

\begin{observation}
The stabilization maps $\Delta_{g,1}\to \Delta_{g+1,1}$ in~\eqref{eq:2} are nullhomotopic. Hence the limiting cohomology vanishes.
\end{observation}

\begin{proof}[Proof sketch]
Recall that the map $\Delta_{g,1}\to \Delta_{g+1,1}$ sends $G \mapsto G \vee S^1$, with total edge
length of $G \subset G \vee S^1$ being $\tfrac{1}{2}$ and
the loop $S^1 \subset G \vee S^1$ also having length $\tfrac{1}{2}$.  Continuously
changing the length distribution from $(\tfrac{1}{2},\tfrac{1}{2})$ to $(0,1)$ defines a
homotopy which starts at the stabilization map and ends at the
constant map $\Delta_{g,1} \to \Delta_{g+1,1}$ sending any
weighted tropical curve $\Gamma$ to a loop of length 1 based at a vertex of weight $g$.
\end{proof}

\begin{remark}\label{rem:rep-stable}
It is natural to ask whether the homology groups $H_*(\Delta_{g,n};\Q)$, viewed as $S_n$-representations, may fit into the framework of \emph{representation stability} from \cite{ChurchFarb13}.  First, if we fix both $k$ and $g$, then $H_k(\Dgn; \Q)$ vanishes for $n \gg 0$.  This follows from the contractibility of $\br$, because there is a natural CW complex structure on $\Dgn / \br$ in which all positive dimensional cells have dimension at least $n - 5g + 5$.  See \cite[Theorem~1.3 and Claim~9.3]{cgp-arxiv}.  This vanishing may be compared with the stabilization with respect to marked points for homology of the pure mapping class group, which says that, for fixed $g$ and $k$, the sequence $H_k(\cM_{g,n}; \Q)$ is representation stable \cite{JimenezRolland11}.

Now suppose we fix $g$ and a small {\em codegree} $k$ and study the sequence of $S_n$-representations $H_{3g-4+n-k}(\Dgn; \Q)$.  We still do not expect representation stability to hold in general: as shown in \cite{ChurchFarb13}, representation stability implies polynomial dimension growth, whereas already the dimension of $H_{n-1}(\Delta_{1,n}; \Q)$ grows super exponentially with $n$.

Nevertheless, 
Wiltshire-Gordon points out that $H_{n-1}(\Delta_{1,n}; \Q)$ admits a natural filtration whose graded pieces are representation stable.  Contracting the repeated marking locus gives a homotopy equivalence between $\Delta_{1,n}$ and the one point compactification of a disjoint union of $(n-1)!/2$ open balls.  These balls are the connected components of the configuration space of $n$ distinct labeled points on a circle, up to rotation and reflection.  There is then a natural identification of $H_{n-1}(\Delta_{1,n}; \Q) \otimes \mathrm{sgn}$ with $H^0$ of this configuration space.  By \cite{VarchenkoGelfand87, Moseley17, MoseleyProudfootYoung17}, this $H^0$ carries a natural filtration, induced by localization on a larger configuration space with $S^1$-action whose graded pieces are finitely generated FI-modules.  
\end{remark}

\appendix
\section{Calculations for $g\ge2$}\label{sec:compute}
We now present some calculations of $\widetilde{H}_{*}(\Delta_{g,n};\Q)$ for $g\ge 2.$  Apart from some small cases, these were carried out by computer using the cellular homology theory for $\Dgn$ as a symmetric $\Delta$-complex.  This is notably more efficient than other available methods, e.g.,~equipping $\Delta_{g,n}$ with a cell structure via barycentric subdivision.  We further simplified the computer calculations via relative cellular homology and the contractibility of subcomplexes given by  Theorem~\ref{thm:contractible}. We also used the program \texttt{boundary} \cite{maggiolo-pagani} which efficiently enumerates symmetric orbits of boundary strata of $\ocM_{g,n}$, and hence of cells in $\Delta_{g,n}$.  By~\eqref{eq:mgn-comparison}, these calculations detect top weight cohomology groups $\Gr^W_{6g-6+2n} H^* (\cM_{g,n};\Q).$

In the case $n=0$, our calculations replicate those from earlier manuscripts of Bar-Natan and McKay \cite{BarNatanMcKay}, given the identification $H_*(\Delta_{g};\Q) \cong H_{*-2g+1}(G^{(g)})$. We refer to that manuscript for further remarks on homology computations for the basic graph complex.  When $n>0$ there is no reason the computations of $H_*(\Dgn;\Q)$ could not have been performed earlier, but since we are currently unaware of an appropriate reference, we include them in Table~\ref{tab:calculations}.  Closely related computations that do appear in the literature, such as those in \cite{KhoroshkinWillwacherZivkovic16}, involve graphs with unlabeled marked points.  
The computations in the case $g = 2$ were also given in \cite{Chan15}, where it is also proved that $\widetilde{H}_*(\Delta_{2,n};\Z)$ is supported in the top two degrees. More recently, they were achieved $S_n$-equivariantly in \cite{Yun20}.

\begin{table}[h]%
\begin{tabular}{r|l}
\hline
$(g,n)$ & Reduced Betti numbers of $\Delta_{g,n}$ for $i=0,\ldots,3g-4+n$ \\
\hline 
$(2,0)$ & $(0,0,0)$\\
$(2,1)$ & $(0,0,0,0)$\\
$(2,2)$ & $(0,0,0,0,1)$\\
$(2,3)$ & $(0,0,0,0,0,0)$\\
$(2,4)$ & $(0,0,0,0,0,1,3)$\\
$(2,5)$ & $(0,0,0,0,0,0,5,15)$\\
$(2,6)$ & $(0,0,0,0,0,0,0,26,86)$\\
$(2,7)$ & $(0,0,0,0,0,0,0,0, 155,575)$\\
$(2,8)$ & $(0,0,0,0,0,0,0,0,0, 1066, 4426)$\\
\hline
$(3,0)$ & $(0, 0, 0, 0, 0, 1)$\\
$(3,1)$ & $(0, 0, 0, 0, 0, 1, 0)$\\
$(3,2)$ & $(0, 0, 0, 0, 0, 0, 0, 0)$\\
$(3,3)$ & $(0, 0, 0, 0, 0, 0, 0, 0, 1)$\\
$(3,4)$ & $(0, 0, 0, 0, 0, 0, 0, 0, 3, 2)$\\
\hline
$(4,0)$ & $(0, 0, 0, 0, 0, 0, 0, 0, 0)$\\
$(4,1)$ & $(0, 0, 0, 0, 0, 0, 0, 0, 0, 0)$\\
$(4,2)$ & $(0, 0, 0, 0, 0, 0, 0, 0, 0, 0, 0)$\\
$(4,3)$ & $(0, 0, 0, 0, 0, 0, 0, 0, 0, 0, 2, 1)$\\
\hline
$(5,0)$ & $(0, 0, 0, 0, 0, 0, 0, 0, 0, 1, 0, 0)$\\
$(5,1)$ & $(0, 0, 0, 0, 0, 0, 0, 0, 0, 1, 0, 0, 0)$\\
\hline
$(6,0)$ & $(0, 0, 0, 0, 0, 0, 0, 0, 0, 0, 0, 0, 0, 0, 1)$\\
\hline
\end{tabular}
\vspace{.5cm}
\caption{For each $(g,n)$ shown, the dimensions of $\widetilde{H}_{i-1}(\Delta_{g,n};\Q)$ for $i=1,\ldots,3g-3+n$.}
\label{tab:calculations}
\end{table}

Some of the homology classes displayed in Table~\ref{tab:calculations} have representatives with small enough support that it is feasible to describe them explicitly.
For instance, for $(g,n) = (2,2)$, it is easy to explicitly describe the unique nonzero homology class in $\Delta_{2,2}$; it is represented by the graph shown in Figure~\ref{f:delta22}.  Every edge of the graph is contained in a triangle, so the graph-theoretic lemma below shows immediately that it is a cycle in homology. Moreover it is obviously nonzero since it is in top degree.

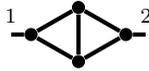
\begin{figure}
\begin{tikzpicture}[my_node/.style={fill, circle, inner sep=1.75pt}, scale=.9]
\begin{scope}[shift={(-5,0.2)}]
\node[my_node] (1) at (0,-.8){};
\node[my_node] (4) at (0,0){};
\node[my_node] (6) at (0.7,-.4){};
\node[my_node] (7) at (-0.7,-.4){};
\node[inv,label={$\scriptstyle 2$}] (mark2) at (1,-.4){};
\node[inv,label={$\scriptstyle 1$}] (mark1) at (-1,-.4){};
\draw[ultra thick] (1)--(4)--(6)--(1)--(7)--(4);
\draw[ultra thick] (6)--(mark2);
\draw[ultra thick] (7)--(mark1);
\end{scope}
\end{tikzpicture}
\caption{The graph appearing in the unique nonzero reduced homology class in $\Delta_{2,2}$.}
\label{f:delta22}
\end{figure}

\begin{lemma}\label{lem:covered-by-triangles}
If $\Gmw\in \J_{g,n}$ has the property that every edge is contained in a triangle, then $\Gmw$ represents a rational cycle in $\Delta_{g,n}$. 
\end{lemma}

\begin{proof}
The boundary of $\Gmw$ in the cellular chain complex $ C_*(\Delta_{g,n};\Q)$ is a sum, with appropriate signs, of 1-edge contractions of $\Gmw$. Each such contraction has parallel edges and hence a non-alternating automorphism, so is zero as a cellular chain.
\end{proof}

For $(g,n)=(3,3)$ and $(6,0)$, the unique nonzero homology group is in top degree, so there is a unique nontrivial cycle, up to scaling.  We have explicit descriptions of these cycles, as linear combinations of trivalent graphs, as shown in Figures~\ref{f:genus33} and \ref{f:genus6}.

For $(g,n)=(3,0),(3,1),(5,0),$ and $(5,1)$, the unique nonzero homology groups in Table~\ref{tab:calculations} are spanned by the classes of ``wheel graphs". Given any $g$, let $W_g$ be a genus $g$ {\em wheel}: the graph obtained from a $g$-cycle $C_g$ by adding a vertex $w$ that is simply adjacent to each vertex of $C_g$.  See Figure~\ref{fig:wheel}.

\begin{figure}[h]
\begin{tikzpicture}[my_node/.style={fill, circle, inner sep=1.75pt}, scale=1]
\def\R{.587} 
\def\H{.2} 
\def\Ratio{1.42} 
\def\M{1.2}
\begin{scope}[scale=1]
\node[my_node] (1) at (0,1){};
\node[my_node] (5) at (.95,.31){};
\node[my_node] (3) at (0.59,-.81){};
\node[my_node] (4) at (-.59,-.81){};
\node[my_node] (2) at (-.95,.31){};
\node[my_node] (W) at (0,0){};
\draw[ultra thick] (1)--(W)--(2)--(W)--(3)--(W)--(4)--(W)--(5);
\draw[ultra thick] (1)--(5)--(3)--(4)--(2)--(1);
\end{scope}
\end{tikzpicture}
\caption{The graph $W_5$.}\label{fig:wheel}
\end{figure}
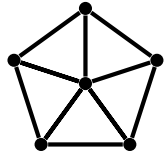
 
\noindent We also regard $W_g$ as an object of $\J_{g,0}$ and let $W_g'$ be the object of $\J_{g,1}$ obtained by marking $w$, the central vertex.
$W_g$ and $W_g'$ represent cells of degree $2g-1$ in $\Delta_{g}$ and $\Delta_{g,1}$, respectively.  When $g$ is even $W_g$ and $W'_g$ have automorphisms that act by odd permutations on the edges, and hence are zero as cellular chains.
When $g$ is odd, these graphs do not have automorphisms that act by odd permutations on the edges, and Lemma~\ref{lem:covered-by-triangles} implies immediately that $W_g$ and $W_g'$ represent rational cycles on $\Delta_{g}$ and $\Delta_{g,1}$ respectively.   Moreover, they are nonzero.  
\begin{lemma}
For $g\ge3 $ odd, the classes represented by $W_g$ and $W_g'$ are nonzero in $\widetilde H_{2g-1}(\Delta_g;\Q)$ and $\widetilde H_{2g-1}(\Delta_{g,1};\Q)$ respectively.
\end{lemma}
\begin{proof} 
  The fact that $W_g$ represents a nontrivial class is established by \cite{Willwacher15}; see \cite[Theorem~2.6]{CGP1-JAMS}.  As for $W_g'$,
  we apply the chain map $\pi: C_*(\Delta_{g,1};\Q) \to C_*(\Delta_g;\Q)$ from the proof of Theorem~\ref{thm:transfer}.  It sends the cycle $W'_g$ to $\pm W_g$, so the homology class represented by $W'_g$ is also nontrivial.
\end{proof}

\begin{figure}[h!]
\begin{tikzpicture}[my_node/.style={fill, circle, inner sep=1.75pt}, scale=.9]
\begin{scope}[shift={(-5,0.2)}]
\node[my_node] (1) at (0,-.8){};
\node[my_node] (2) at (-.7,.5){};
\node[my_node] (3) at (.7,.5){};
\node[my_node] (4) at (0,0){};
\node[my_node] (5) at (0,.9){};
\node[my_node] (6) at (0.7,-.4){};
\node[my_node] (7) at (-0.7,-.4){};
\node[inv,label={$\scriptstyle 1$}] (mark1) at (0,1.3){};
\node[inv,label={$\scriptstyle 2$}] (mark2) at (1,-.6){};
\node[inv,label={$\scriptstyle 3$}] (mark3) at (-1,-.6){};
\draw[ultra thick] (1)--(4)--(3)--(6)--(1)--(7)--(2)--(5)--(3)--(4)--(2);
\draw[ultra thick] (6)--(mark2);
\draw[ultra thick] (7)--(mark3);
\draw[ultra thick] (5)--(mark1);
\end{scope}
\begin{scope}[shift={(-2.8,0)},scale=1.2]
\node[my_node] (1) at (0,0){};
\node[my_node] (2) at (0,0.8){};
\node[my_node] (3) at (0.8,0.8){};
\node[my_node] (4) at (0.8,0){};
\node[my_node] (5) at (1.2,0.4){};
\node[my_node] (6) at (-.4,0.7){};
\node[my_node] (7) at (-.4,0.1){};
\node[inv,label={$\scriptstyle 1$}] (mark1) at (1.4,.4){};
\node[inv,label={$\scriptstyle 2$}] (mark2) at (-.6,.7){};
\node[inv,label={$\scriptstyle 3$}] (mark3) at (-.6,.1){};
\draw[ultra thick] (1)--(4)--(3)--(2)--(7)--(1)--(6)--(2)--(3)--(5)--(4);
\draw[ultra thick] (6)--(mark2);
\draw[ultra thick] (7)--(mark3);
\draw[ultra thick] (5)--(mark1);
\end{scope}
\begin{scope}[shift={(0,0)},scale=1.2]
\node[my_node] (1) at (0,0){};
\node[my_node] (2) at (0,0.8){};
\node[my_node] (3) at (0.8,0.8){};
\node[my_node] (4) at (0.8,0){};
\node[my_node] (5) at (1.2,0.4){};
\node[my_node] (6) at (-.4,0.7){};
\node[my_node] (7) at (-.4,0.1){};
\node[inv,label={$\scriptstyle 2$}] (mark2) at (1.4,.4){};
\node[inv,label={$\scriptstyle 1$}] (mark1) at (-.6,.7){};
\node[inv,label={$\scriptstyle 3$}] (mark3) at (-.6,.1){};
\draw[ultra thick] (1)--(4)--(3)--(2)--(7)--(1)--(6)--(2)--(3)--(5)--(4);
\draw[ultra thick] (6)--(mark1);
\draw[ultra thick] (7)--(mark3);
\draw[ultra thick] (5)--(mark2);
\end{scope}
\begin{scope}[shift={(3,0)},scale=1.2]
\node[my_node] (1) at (0,0){};
\node[my_node] (2) at (0,0.8){};
\node[my_node] (3) at (0.8,0.8){};
\node[my_node] (4) at (0.8,0){};
\node[my_node] (5) at (1.2,0.4){};
\node[my_node] (6) at (-.4,0.7){};
\node[my_node] (7) at (-.4,0.1){};
\node[inv,label={$\scriptstyle 3$}] (mark3) at (1.4,.4){};
\node[inv,label={$\scriptstyle 1$}] (mark1) at (-.6,.7){};
\node[inv,label={$\scriptstyle 2$}] (mark2) at (-.6,.1){};
\draw[ultra thick] (1)--(4)--(3)--(2)--(7)--(1)--(6)--(2)--(3)--(5)--(4);
\draw[ultra thick] (6)--(mark1);
\draw[ultra thick] (7)--(mark2);
\draw[ultra thick] (5)--(mark3);
\end{scope}
\begin{scope}[shift={(6,.2)},scale=1]
\node[my_node] (1) at (.5,0){};
\node[my_node] (2) at (-.5,0){};
\node[my_node] (3) at (.5,1){};
\node[my_node] (4) at (-.5,1){};
\node[my_node] (5) at (-.7,.5){};
\node[my_node] (6) at (.7,.5){};
\node[my_node] (7) at (0,-.4){};
\node[inv,label={$\scriptstyle 1$}] (mark1) at (-1,.5){};
\node[inv,label={$\scriptstyle 2$}] (mark2) at (1,.5){};
\node[inv,label={0:$\scriptstyle 3$}] (mark3) at (0,-.7){};
\draw[ultra thick] (1)--(6)--(3)--(4)--(5)--(2)--(3)--(2)--(7)--(1)--(4)--(5)--(2);
\draw[ultra thick] (5)--(mark1);
\draw[ultra thick] (6)--(mark2);
\draw[ultra thick] (7)--(mark3);
\end{scope}
\begin{scope}[shift={(-5,-2.5)},scale=1]
\node[my_node] (1) at (.5,0){};
\node[my_node] (2) at (-.5,0){};
\node[my_node] (3) at (.5,1){};
\node[my_node] (4) at (-.5,1){};
\node[my_node] (5) at (-.7,.5){};
\node[my_node] (6) at (.7,.5){};
\node[my_node] (7) at (0,-.4){};
\node[inv,label={$\scriptstyle 1$}] (mark1) at (-1,.5){};
\node[inv,label={$\scriptstyle 3$}] (mark3) at (1,.5){};
\node[inv,label={0:$\scriptstyle 2$}] (mark2) at (0,-.7){};
\draw[ultra thick] (1)--(6)--(3)--(4)--(5)--(2)--(3)--(2)--(7)--(1)--(4)--(5)--(2);
\draw[ultra thick] (5)--(mark1);
\draw[ultra thick] (6)--(mark3);
\draw[ultra thick] (7)--(mark2);
\end{scope}
\begin{scope}[shift={(-2.5,-2.5)},scale=1]
\node[my_node] (1) at (.5,0){};
\node[my_node] (2) at (-.5,0){};
\node[my_node] (3) at (.5,1){};
\node[my_node] (4) at (-.5,1){};
\node[my_node] (5) at (-.7,.5){};
\node[my_node] (6) at (.7,.5){};
\node[my_node] (7) at (0,-.4){};
\node[inv,label={$\scriptstyle 2$}] (mark2) at (-1,.5){};
\node[inv,label={$\scriptstyle 3$}] (mark3) at (1,.5){};
\node[inv,label={0:$\scriptstyle 1$}] (mark1) at (0,-.7){};
\draw[ultra thick] (1)--(6)--(3)--(4)--(5)--(2)--(3)--(2)--(7)--(1)--(4)--(5)--(2);
\draw[ultra thick] (5)--(mark2);
\draw[ultra thick] (6)--(mark3);
\draw[ultra thick] (7)--(mark1);
\end{scope}
\begin{scope}[shift={(0.3,-1.7)},scale=1.4]
\node[my_node] (1) at (.5,0){};
\node[my_node] (2) at (-.5,0){};
\node[my_node] (3) at (-.4,-.5){};
\node[my_node] (4) at (.4,-.5){};
\node[my_node] (5) at (-.7,-.25){};
\node[my_node] (6) at (.7,-.25){};
\node[my_node] (7) at (0,0){};
\node[inv,label={$\scriptstyle 2$}] (mark2) at (1,-.25){};
\node[inv,label={$\scriptstyle 1$}] (mark1) at (-1,-.25){};
\node[inv,label={$\scriptstyle 3$}] (mark3) at (0,.3){};
\draw[ultra thick] (1)--(6)--(4)--(1)--(7)--(2)--(5)--(3)--(4)--(3)--(2);
\draw[ultra thick] (5)--(mark1);
\draw[ultra thick] (6)--(mark2);
\draw[ultra thick] (7)--(mark3);
\end{scope}
\begin{scope}[shift={(3.5,-1.7)},scale=1.4]
\node[my_node] (1) at (.5,0){};
\node[my_node] (2) at (-.5,0){};
\node[my_node] (3) at (-.4,-.5){};
\node[my_node] (4) at (.4,-.5){};
\node[my_node] (5) at (-.7,-.25){};
\node[my_node] (6) at (.7,-.25){};
\node[my_node] (7) at (0,0){};
\node[inv,label={$\scriptstyle 3$}] (mark3) at (1,-.25){};
\node[inv,label={$\scriptstyle 1$}] (mark1) at (-1,-.25){};
\node[inv,label={$\scriptstyle 2$}] (mark2) at (0,.3){};
\draw[ultra thick] (1)--(6)--(4)--(1)--(7)--(2)--(5)--(3)--(4)--(3)--(2);
\draw[ultra thick] (5)--(mark1);
\draw[ultra thick] (6)--(mark3);
\draw[ultra thick] (7)--(mark2);
\end{scope}
\begin{scope}[shift={(6.7,-1.7)},scale=1.4]
\node[my_node] (1) at (.5,0){};
\node[my_node] (2) at (-.5,0){};
\node[my_node] (3) at (-.4,-.5){};
\node[my_node] (4) at (.4,-.5){};
\node[my_node] (5) at (-.7,-.25){};
\node[my_node] (6) at (.7,-.25){};
\node[my_node] (7) at (0,0){};
\node[inv,label={$\scriptstyle 3$}] (mark3) at (1,-.25){};
\node[inv,label={$\scriptstyle 2$}] (mark2) at (-1,-.25){};
\node[inv,label={$\scriptstyle 1$}] (mark1) at (0,.3){};
\draw[ultra thick] (1)--(6)--(4)--(1)--(7)--(2)--(5)--(3)--(4)--(3)--(2);
\draw[ultra thick] (5)--(mark2);
\draw[ultra thick] (6)--(mark3);
\draw[ultra thick] (7)--(mark1);
\end{scope}

\end{tikzpicture}
\caption{The graphs appearing in the unique nonzero reduced homology class in $\Delta_{3,3}$, with unsigned coefficients $1, 1, 1, 1, 1, 1, 1, 2, 2, 2$.}
\label{f:genus33}
\end{figure}
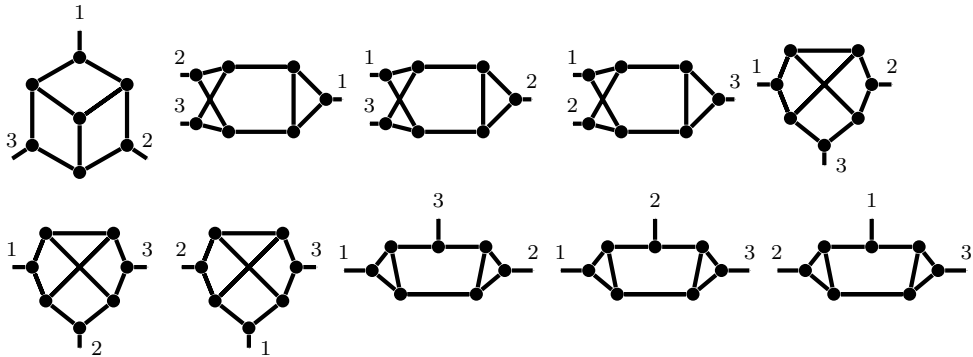

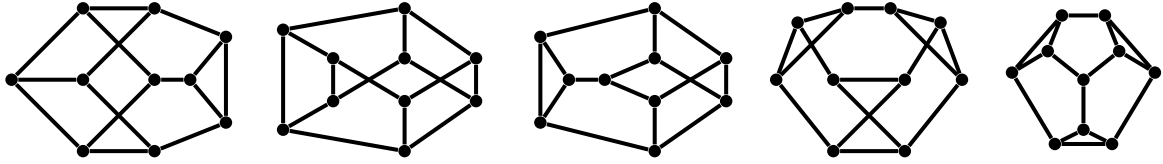
\begin{figure} [h!] \scalebox{.95}{
\begin{tikzpicture}[my_node/.style={fill, circle, inner sep=1.75pt}, scale=1]
\begin{scope}[shift={(-1.5,0)}]
\node[my_node] (A) at (0,0){};
\node[my_node] (B) at (1,1){};
\node[my_node] (C) at (1,0){};
\node[my_node] (D) at (1,-1){};
\node[my_node] (E) at (2,1){};
\node[my_node] (F) at (2,0){};
\node[my_node] (G) at (2,-1){};
\node[my_node] (H) at (3,.6){};
\node[my_node] (I) at (2.5,0){};
\node[my_node] (J) at (3,-.6){};
\draw[ultra thick] (A)--(B)--(E)--(H)--(J)--(G)--(D)--(A);
\draw[ultra thick] (B)--(F)--(D);
\draw[ultra thick] (A)--(C)--(E);
\draw[ultra thick] (C)--(G);
\draw[ultra thick] (F)--(I)--(H);
\draw[ultra thick] (I)--(J);
\end{scope}
\begin{scope}[shift = {(2,0)}]
\node[my_node] (A) at (0.3,.7){};
\node[my_node] (C) at (1,-.3){};
\node[my_node] (B) at (1,.3){};
\node[my_node] (D) at (0.3,-.7){};
\node[my_node] (E) at (2,1){};
\node[my_node] (F) at (2,.3){};
\node[my_node] (G) at (3,.3){};
\node[my_node] (H) at (3,-.3){};
\node[my_node] (I) at (2,-.3){};
\node[my_node] (J) at (2,-1){};
\draw[ultra thick] (A)--(B)--(C)--(D)--(J)--(H)--(G)--(E)--(A);
\draw[ultra thick] (A)--(D);
\draw[ultra thick] (H)--(F);
\draw[ultra thick] (C)--(F)--(E);
\draw[ultra thick] (B)--(I)--(J);
\draw[ultra thick] (I)--(G);
\end{scope}
\begin{scope}[shift = {(5.5,0)}]
\node[my_node] (A) at (.8,0){};
\node[my_node] (B) at (.4,.6){};
\node[my_node] (C) at (.4,-.6){};
\node[my_node] (D) at (1.3,0){};
\node[my_node] (E) at (2,1){};
\node[my_node] (F) at (2,.3){};
\node[my_node] (G) at (2,-.3){};
\node[my_node] (H) at (2,-1){};
\node[my_node] (I) at (3,.3){};
\node[my_node] (J) at (3,-.3){};
\draw[ultra thick] (A)--(B)--(E)--(I)--(J)--(H)--(C)--(A)--(D)--(F)--(E);
\draw[ultra thick] (B)--(C);
\draw[ultra thick] (D)--(G)--(I);
\draw[ultra thick] (G)--(H);
\draw[ultra thick] (F)--(J);
\end{scope}
\begin{scope}[shift = {(10.5,0)}]
\node[my_node] (A) at (-1,.8){};
\node[my_node] (B) at (-1.3,0){};
\node[my_node] (C) at (-.3,1){};
\node[my_node] (D) at (-.5,0){};
\node[my_node] (E) at (-.5,-1){};
\node[my_node] (F) at (.5,.0){};
\node[my_node] (G) at (.5,-1){};
\node[my_node] (H) at (.3,1){};
\node[my_node] (I) at (1,.8){};
\node[my_node] (J) at (1.3,0){};
\draw[ultra thick] (A)--(B)--(E)--(G)--(J)--(I)--(H)--(C)--(A)--(D)--(F)--(I);
\draw[ultra thick] (B)--(C);
\draw[ultra thick] (H)--(J);
\draw[ultra thick] (G)--(D);
\draw[ultra thick] (F)--(E);
\end{scope}
\begin{scope}[shift = {(13.5,-.1)},scale=1]
\node[my_node] (A) at (-1,.2){};
\node[my_node] (B) at (-.5,.5){};
\node[my_node] (C) at (-.3,1){};
\node[my_node] (D) at (.3,1){};
\node[my_node] (E) at (.5,.5){};
\node[my_node] (F) at (1,.2){};
\node[my_node] (G) at (.4,-.8){};
\node[my_node] (H) at (0,-.6){};
\node[my_node] (I) at (-.4,-.8){};
\node[my_node] (J) at (0,0.1){};
\draw[ultra thick] (A)--(B)--(C)--(D)--(E)--(F)--(G)--(H)--(I)--(A)--(C);
\draw[ultra thick] (D)--(F);
\draw[ultra thick] (I)--(G);
\draw[ultra thick] (B)--(J)--(E);
\draw[ultra thick] (J)--(H);
\end{scope}
\end{tikzpicture} }
\caption{The graphs appearing in the unique nonzero reduced homology class in $\Delta_6$, with unsigned coefficients $2,3,6,3,4$.}  
\label{f:genus6}
\end{figure}

\bibliographystyle{amsalpha}
\bibliography{math}

\end{document}